\documentclass{article}
\usepackage{amsmath}
\usepackage{amsfonts,amssymb,mathrsfs, textcomp}
\usepackage{theorem}
\usepackage{hyperref}

\usepackage{tikz}
\usepackage{tikz-cd}
\usepgflibrary{arrows}
\usetikzlibrary{matrix}
\tikzset{>=stealth'} 

\newcommand{\Ob}{\mathrm{Ob}}
\newcommand{\Ih}{\mathrm{Ih}}
\newcommand{\IK}{\mathrm{IK}}

\newcommand{\Core}{\mathrm{Core}}
\newcommand{\lcm}{\mathrm{lcm}}

\newcommand{\into}{\hookrightarrow}
\newcommand{\tto}{\longrightarrow}

\newcommand{\iso}{\overset{\simeq}{\tto}}

\input amssym.def

\numberwithin{equation}{section}

\numberwithin{figure}{section}

\makeatletter
\newcommand\qedsymbol{\hbox{$\Box$}}
\newcommand\qed{\relax\ifmmode\Box\else
  {\unskip\nobreak\hfil\penalty50\hskip1em\null\nobreak\hfil\qedsymbol
  \parfillskip=\z@\finalhyphendemerits=0\endgraf}\fi}

\newenvironment{proof-of}[1][{}]{\par\noindent \textbf{Proof of} {#1}. }{\qed}

\newenvironment{proof}[0]{\par\noindent \textbf{Proof}.}{\qed}

\DeclareMathOperator*{\Gal}{Gal}
\DeclareMathOperator*{\Aff}{Aff}

\newcommand{\GT}{\mathsf{GT}}

\newcommand{\GTSh}{\mathsf{GTSh}}
\newcommand{\GTh}{\widehat{\mathsf{GT}}}
\newcommand{\Zhat}{\widehat{\mathbb{Z}}}

\newcommand{\F}{\mathsf{F}}

\newcommand{\NFI}{\mathsf{NFI}}
\newcommand{\Dih}{\mathsf{Dih}}

\newcommand{\PB}{\mathrm{PB}} 
\newcommand{\B}{\mathrm{B}} 
\renewcommand{\H}{\mathsf{H}} 
\newcommand{\K}{\mathsf{K}} 
\newcommand{\N}{\mathsf{N}} 

\newcommand{\PR}{\mathscr{PR}}
 
\newcommand{\End}{\mathsf{End}}

\newcommand{\Aut}{\mathrm{Aut}}

\newcommand{\ord}{\mathrm{ord}}
\newcommand{\isom}{\mathrm{isom}}

\newcommand{\opp}{\mathrm{opp}}

\newcommand{\conn}{\mathrm{conn}}

\newcommand{\ti}[1]{{\tilde{#1}}}

\newcommand{\wh}[1]{{\widehat{#1}}}
\newcommand{\ol}[1]{{\overline{#1}}}

\newcommand{\e}[1]{{\textbf{#1}}}
\newcommand{\txt}[1]{{\textrm{#1}}}

\newcommand{\hs}{\heartsuit}

\newcommand{\mF}{{\mathfrak{F}}}

\newcommand{\cF}{\mathcal{F}}

\newcommand{\cZ}{\mathcal{Z}}

\newcommand{\cP}{\mathcal{P}}

\newcommand{\cR}{\mathcal{R}}

\newcommand{\cX}{\mathcal{X}}
\newcommand{\ML}{{\mathcal{ML}}}
\newcommand{\hcP}{\widehat{\mathcal{P}}}

\newcommand{\PP}{{\mathbb P}}
\newcommand{\CC}{{\mathbb C}}

\newcommand{\ZZ}{{\mathbb Z}}
\newcommand{\QQ}{{\mathbb Q}}

\newcommand{\al}{{\alpha}}

\newcommand{\vro}{{\varrho}}

\newcommand{\si}{{\sigma}}
\newcommand{\ga}{{\gamma}}
\newcommand{\vf}{{\varphi}}

\newcommand{\varka}{{\varkappa}}

\newcommand{\te}{\theta}

\newcommand{\de}{{\delta}}
\newcommand{\D}{{\Delta}}

\renewcommand{\mod}{~\mathrm{mod}~}

\newcommand{\lan}{\langle\,}
\newcommand{\ran}{\,\rangle}

\date{}
\newtheorem{thm}{Theorem}[section]
\newtheorem{defi}[thm]{Definition}

\newtheorem{lem}[thm]{Lemma}
\newtheorem{cor}[thm]{Corollary}
\newtheorem{conj}[thm]{Conjecture}
\newtheorem{prop}[thm]{Proposition}

\theorembodyfont{\rm}
\newtheorem{remark}[thm]{Remark}

\title{Accessing non-abelian quotients of the Grothendieck-Teichmueller group via elementary tools}

\author{Ivan Bortnovskyi, Vasily A. Dolgushev, Borys Holikov, Vadym Pashkovskyi}

\date{}

\begin{document}

\large

\maketitle

\begin{flushright}
{\it To our teachers with admiration}
\end{flushright}

\bigskip

\begin{abstract}
Many challenging questions about the Grothendieck-Teichmueller group, $\GTh$, are motivated 
by the fact that this group receives the injective homomorphism (called the Ihara embedding) 
from the absolute Galois group, $G_{\mathbb{Q}}$, of rational numbers. Although the question about 
the surjectivity of the Ihara embedding is a very challenging problem, in this paper, we construct a 
family of finite non-abelian quotients of $\GTh$ that receive surjective homomorphisms 
from $G_{\QQ}$. We also assemble these finite quotients into an infinite (non-abelian) 
profinite quotient of $\GTh$. We prove that the natural homomorphism from $G_{\QQ}$
to the resulting profinite group is also surjective. We describe this profinite group explicitly as a clopen 
subgroup of index $2$ in the group of affine transformations associated to the ring of $2$-adic integers.
 
To achieve these goals, we used the groupoid $\GTSh$ of $\GT$-shadows for the gentle version of 
the Grothendieck-Teichmueller group. This groupoid  was introduced in the recent paper 
by the second author and J. Guynee and the set $\Ob(\GTSh)$ of objects of $\GTSh$ is a poset
of certain finite index normal subgroups of the Artin braid group on $3$ strands. We introduce 
a sub-poset $\Dih$ of $\Ob(\GTSh)$ related to the family of dihedral groups and call it the dihedral 
poset. We show that each element $\K$ of $\Dih$ is the only object of its connected 
component in $\GTSh$. Using the surjectivity of the cyclotomic character, we prove 
that, if the order of the dihedral group corresponding to $\K$ is a power of $2$, then the natural 
homomorphism from $G_{\QQ}$ to the finite group $\GTSh(\K, \K)$ is surjective.  
We introduce the Lochak-Schneps conditions on morphisms of $\GTSh$ and prove that 
each morphism of $\GTSh$ with the target $\K \in \Dih$ satisfies the Lochak-Schneps conditions. 
Finally, we conjecture that the natural homomorphism from $G_{\QQ}$ to the finite group 
$\GTSh(\K, \K)$ is surjective for every object $\K$ of the dihedral poset. 
\end{abstract}

\section{Introduction}
\label{sec:intro}
Many questions \cite{LochakSchneps-Open}, \cite{PopSurvey} about the profinite version $\GTh$ of the 
Grothendieck-Teichmueller group introduced in \cite[Section 4]{Drinfeld} are motivated 
by its link to the absolute Galois group $G_{\QQ} := \Gal(\ol{\QQ}/\QQ)$ of rational numbers. 
In his famous paper \cite{Ihara}, Y. Ihara constructed 
a group homomorphism 
\begin{equation}
\label{Ih-intro}
\Ih : G_{\QQ}  \to \GTh.
\end{equation}
Belyi's theorem \cite{Belyi} implies\footnote{See also Theorems 4.7.6, 4.7.7 and 
Fact 4.7.8 in \cite{Szamuely}.} that the homomorphism $\Ih$ is injective and we call 
it the Ihara embedding. Although the question of surjectivity of $\Ih$ (see \cite{Ihara-ICM}) is
a very challenging one, in this paper, we use rather elementary tools to 
produce examples of non-abelian quotients 
of $\GTh$ that receive surjective homomorphisms from $G_{\QQ}$.
More precisely, we construct a family of finite non-abelian quotients of 
$\GTh$ using the groupoid $\GTSh$ of $\GT$-shadows for 
the gentle version\footnote{As far as we know, the group $\GTh_{gen}$ 
was introduced in paper \cite{Harbater-Schneps} by D. Harbater and L. Schneps.
In their paper, the authors denote $\GTh_{gen}$ by $\GTh_0$.} 
$\GTh_{gen}$ of the Grothendieck-Teichmueller group $\GTh$
\cite{GTgentle}, \cite{JXthesis}. Using this family, we also construct 
an infinite non-abelian profinite quotient of $\GTh$ that receives a surjective 
homomorphism from $G_{\QQ}$. 

Let $\Zhat$ (resp. $\wh{\F}_2$) be the profinite completion of the ring $\ZZ$ 
(resp. the free group $\F_2$ on two generators). The group $\GTh_{gen}$ consists of 
pairs $(\hat{m}, \hat{f}) \in \Zhat \times \wh{\F}_2$ satisfying the hexagon relations 
(see equations (1.1) and (1.2) in \cite[Introduction]{GTgentle}) and the technical condition 
$\hat{f} \in [\wh{\F}_2, \wh{\F}_2]^{top. cl.}$.
For the definition of the multiplication in $\GTh_{gen}$, we refer the reader to 
\cite[Section 2]{GTgentle}, \cite[Section 1.1]{Leila-survey}, \cite[Section 3.1]{JXthesis}.

The original version $\GTh$ of the Grothendieck-Teichmueller group \cite[Section 4]{Drinfeld}
is a subgroup of $\GTh_{gen}$. Hence $\GTh_{gen}$ also receives an injective homomorphism 
from $G_{\QQ}$ and we use the same notation $\Ih$ for the homomorphism from $G_{\QQ}$ to 
$\GTh_{gen}$.

\subsection{The groupoid $\GTSh$ of $\GT$-shadows in a nutshell}
Let $\B_3$ be the Artin braid group \cite{Birman}, \cite{Braids} on $3$ strands:
$$
\B_3 := \lan \si_1, \si_2 ~|~  \si_1 \si_2 \si_1 =  \si_2 \si_1 \si_2 \ran 
$$
and $\PB_3$ be the kernel of the standard homomorphism $\rho$ from $\B_3$ to the symmetric 
group $S_3$ on $3$ letters. It is known \cite[Section 1.3]{Braids} that 
$$
\PB_3 \cong \lan  x_{12}, x_{23} \ran \times \lan c \ran,  
$$
where $x_{12} := \si_1^2$, $x_{23} := \si_2^2$ and $c : =  (\si_1 \si_2 \si_1)^2$.  
We identify the free group $\F_2$ on 
two generators with the subgroup $\lan  x_{12}, x_{23} \ran$ of $\PB_3$ generated 
by $x_{12}$ and $x_{23}$.

Just as in \cite{GTgentle}, we denote by $\GTSh$ the groupoid 
whose set  $\Ob(\GTSh)$ of objects is the poset 
$$
\NFI_{\PB_3}(\B_3) : = \{\N \unlhd \B_3 ~|~ \N \le \PB_3, ~~ |\PB_3: \N| < \infty \}.  
$$ 

For $\N \in \NFI_{\PB_3}(\B_3)$, we consider the finite set 
\begin{equation}
\label{ZZ-mod-Nord-F2-mod-NF2}
\ZZ/ N_{\ord} \ZZ ~\times~ \F_2 / \N_{\F_2}\,,
\end{equation}
where $\N_{\F_2} : =  \N \cap \lan  x_{12}, x_{23} \ran$ and 
$N_{\ord}$ is the least common multiple of the orders of the cosets 
$x_{12} \N$,  $x_{23} \N$ and $c \N$ in the finite group $\PB_3/\N$. 

We denote by $\GT(\N)$ the set of morphisms of the groupoid 
$\GTSh$ with the target $\N$. These are elements of the finite set 
\eqref{ZZ-mod-Nord-F2-mod-NF2} that satisfy the hexagon relations 
(see \eqref{eq:fhex}, \eqref{eq:shex}) modulo $\N$ and additional technical 
conditions. We call morphisms of the groupoid $\GTSh$ 
$\GT$-{\bf shadows}.

Let $(m, f)$ be a pair in $\ZZ \times \F_2$ that represents a $\GT$-shadow 
with the target $\N$. Hexagon relations \eqref{eq:fhex}, \eqref{eq:shex} imply that
the formulas 
$$
T_{m, f} (\si_1) := \si^{2m + 1}_1 \, \N, \qquad 
T_{m, f} (\si_2) := f^{-1} \si^{2m + 1}_2 f \, \N
$$
define a group homomorphism $T_{m, f} : \B_3 \to \B_3/\N$. 
It is convenient to denote by $[m, f]$ the element of $\GT(\N)$ represented 
by a pair $(m, f) \in \ZZ \times \F_2$.  

For $\K, \N \in \NFI_{\PB_3}(\B_3)$, the set $\GTSh(\K, \N)$ of morphisms in $\GTSh$ from 
$\K$ to $\N$ is 
$$
\GTSh(\K, \N) : = \{[m, f]  \in \GT(\N)~|~ \ker(T_{m, f}) = \K \}. 
$$ 

For the definition of the composition of morphisms in $\GTSh$, we refer the reader 
to Theorem \ref{thm:GTSh} of this paper (see also \cite[Theorem 3.10]{GTgentle}).  

The groupoid $\GTSh$ is highly disconnected. Indeed, due to Proposition 
\ref{prop:quotients-isomorphic}, if $\GTSh(\K, \N)$ is non-empty, then the quotient 
groups $\B_3/\N$ and $\B_3/\K$ are isomorphic. However, using the finiteness of the 
set $\GT(\N)$, it is easy to show that, for every $\N \in  \NFI_{\PB_3}(\B_3)$, the connected 
component $\GTSh_{\conn}(\N)$ of $\N$ in $\GTSh$ is a finite groupoid. 

It is certainly easier to work with a connected component of $\GTSh$ that 
has exactly one object. Thus, if $\N$ is the only object of its connected component 
$\GTSh_{\conn}(\N)$, then we say that $\N$ is an {\bf isolated object} of $\GTSh$. 
In this case, $\GT(\N) = \GTSh(\N, \N)$ and hence $\GT(\N)$ is a (finite) group.  

Let $\H, \K$ be elements of $\NFI_{\PB_3}(\B_3)$ such that $\H \le \K$. Furthermore, let 
$(m, f)$ be a pair in $\ZZ \times \F_2$ that represents a $\GT$-shadow with the target $\H$.
Due to \cite[Proposition 3.12]{GTgentle}, the same pair $(m, f)$ also represents a $\GT$-shadow
with the target $\K$ and we get a natural map 
$$
\cR_{\H, \K} : \GT(\H) \to \GT(\K). 
$$
It is not hard to show that, if $\H \le \K$ are isolated objects of $\GTSh$, then the map $\cR_{\H, \K}$
is a group homomorphism. In this paper, we call $\cR_{\H, \K}$ the {\bf reduction map} and, sometimes, 
the {\bf reduction homomorphism}. 

\subsection{A link between $\GTh_{gen}$ and the groupoid $\GTSh$}
\label{link}
For a group $G$ and a finite index normal subgroup $\N \unlhd G$, 
we denote by $\hcP_{\N}$ the standard group homomorphism 
$$
\hcP_{\N} : \wh{G} \to G/\N
$$
from the profinite completion $\wh{G}$ of $G$ to the finite group $G/\N$. Moreover, for a positive 
integer $N$, we set $\hcP_{N} : = \hcP_{N \ZZ}$, i.e. $\hcP_{N}$ is the standard ring
homomorphism from $\Zhat$ to the finite ring $\ZZ / N\ZZ$. 

Given $\N \in \NFI_{\PB_3}(\B_3)$ and $(\hat{m}, \hat{f}) \in \GTh_{gen}$, the pair 
$$
\big( \hcP_{N_{\ord}}(\hat{m}), \hcP_{\N_{\F_2}}(\hat{f}) \big)
$$
is a $\GT$-shadow with the target $\N$. In other words, the formula 
\begin{equation}
\label{PR-N}
\PR_{\N}(\hat{m}, \hat{f}) := \big( \hcP_{N_{\ord}}(\hat{m}), \hcP_{\N_{\F_2}}(\hat{f}) \big)
\end{equation} 
defines a natural map $\PR_{\N} :  \GTh_{gen} \to \GT(\N)$. If a $\GT$-shadow $[m,f] \in \GT(\N)$ 
belongs to the image of $\PR_{\N}$, then we say that $[m,f]$ is {\bf genuine}; otherwise 
$[m, f]$ is called {\bf fake}.

Due to \cite[Corollary 5.4]{GTgentle}, a $\GT$-shadow $[m, f] \in \GT(\N)$ is genuine if and only if 
$[m, f]$ belongs to the image of the reduction map $\cR_{\H, \N}$ for every $\H \in \NFI_{\PB_3}(\B_3)$ 
such that $\H \le \N$. At the time of writing, the authors do not know a single
example of a fake $\GT$-shadow.  

\begin{remark}  
\label{rem:ML}
Using the reduction maps, one can construct \cite[Section 5]{GTgentle} a functor $\ML$ from the subposet 
$$
\NFI^{isolated}_{\PB_3}(\B_3) \subset \NFI_{\PB_3}(\B_3)
$$
of isolated objects of the groupoid $\GTSh$ to the category of finite groups. Moreover, one can show 
\cite{GTgentle} that the natural group homomorphism $\GTh_{gen} \to \lim(\ML)$ is an isomorphism 
of (topological) groups. 
\end{remark}  
\begin{remark}  
\label{rem:WhatAre}
$\GT$-shadows for the original version of $\GTh$ \cite[Section 4]{Drinfeld}
were introduced in paper \cite{GTShadows}. Note that, in paper \cite{GTShadows}, 
the notation $\GTSh$ is used for the groupoid of $\GT$-shadows for $\GTh$ 
and the set of objects of this groupoid is $\NFI_{\PB_4}(\B_4)$. In this paper, 
$\GTSh$ denotes the groupoid of $\GT$-shadows for $\GTh_{gen}$ and, here, 
$\Ob(\GTSh) := \NFI_{\PB_3}(\B_3)$. 
\end{remark}

\subsection{The Ihara homomorphism and its version for $\GT$-shadows}
\label{sec:Ihara-Ih-N}

Recall \cite{SGA1}, \cite[Section 4.7]{Szamuely} that we have the short exact sequence
of profinite groups: 
\begin{equation}
\label{ses-pi1-GQQ}
1 \tto 
\pi_1\big(\PP^1_{\ol{\QQ}} -  \{0,1,\infty\}\big) 
\tto
\pi_1\big(\PP^1_{\QQ} - \{0,1,\infty\}\big) 
\tto G_{\QQ} \to 1,
\end{equation}
where $\pi_1\big(\PP^1_{\ol{\QQ}} -  \{0,1,\infty\}\big)$ (resp. 
$\pi_1\big(\PP^1_{\QQ} -  \{0,1,\infty\}\big)$) is the algebraic fundamental 
group of the variety $\PP^1_{\ol{\QQ}} -  \{0,1,\infty\}$ (resp. the variety 
$\PP^1_{\QQ} -  \{0,1,\infty\}$). 

In \cite[Section 1]{Ihara}, Y. Ihara constructed a splitting of \eqref{ses-pi1-GQQ} 
that gives us an action of $G_{\QQ}$ on $\wh{\F}_2$ of the form 
$$
g(x)  = x^{\chi(g)}, \qquad g(y)  = f_g^{-1} y^{\chi(g)} f_g, \qquad g \in G_{\QQ}\,,
$$
where $\chi: G_{\QQ} \to \wh{\ZZ}^{\times}$ is the cyclotomic character and the 
construction of the element $f_g \in \wh{\F}_2$ is described in great detail in \cite[Section 1.4]{Ihara}.
(See also \cite[Corollary 4.7.3]{Szamuely} and \cite[Example 4.7.4]{Szamuely}.)
 
It is known \cite[Section 4]{Drinfeld}, \cite[Section 1]{Ihara},  
\cite[Theorem 4.7.7]{Szamuely}, \cite[Fact 4.7.8]{Szamuely} that, 
\begin{itemize}
 
\item the formula
\begin{equation}
\label{Ih-formula}
\Ih(g) := \big(\, (\chi(g)-1)/2, f_g \big)  \in \wh{\ZZ} \times \wh{F}_2 
\end{equation}
defines a group homomorphism 
\begin{equation}
\label{GQ-to-GTh}
G_{\QQ} ~\to~ \GTh
\end{equation}
and 
 
\item Belyi's theorem \cite{Belyi} implies that homomorphism \eqref{GQ-to-GTh} is injective.
 
\end{itemize}
In this paper, we call the homomorphism in \eqref{GQ-to-GTh} the \e{Ihara embedding}.
Since $\GTh$ is a subgroup of $\GTh_{gen}$, formula \eqref{Ih-formula}
also defines an injective homomorphism $G_{\QQ} \into \GTh_{gen}$. We use the same notation 
$\Ih$ for the injective homomorphism $G_{\QQ} \into \GTh_{gen}$.  

Recall \cite[Remark 2.9]{GTgentle} that the formula
\begin{equation}
\label{chi-virtual}
\chi_{vir}(\hat{m}, \hat{f}) : = 2 \hat{m} + 1 
\end{equation}
defines a group homomorphism $\chi_{vir} : \GTh_{gen} \to \wh{\ZZ}^{\times}$
and it is called the \e{virtual cyclotomic character}. 

\begin{remark}
\label{rem:chi-virtual-onto}
Using \eqref{Ih-formula} and \eqref{chi-virtual}, it is easy to see that the diagram 
\begin{equation}
\label{chi-chi-virtual}
\begin{tikzpicture}
\matrix (m) [matrix of math nodes, row sep=1.5em, column sep=1.5em]
{G_{\QQ}  &  ~ & \GTh_{gen} \\
~ &  ~~\wh{\ZZ}^{\times} &~ \\};
\path[->, font=\scriptsize]
(m-1-1) edge node[above] {$\Ih$} (m-1-3)
edge node[left] {$\chi~$} (m-2-2)  
(m-1-3) edge node[right] {$~\chi_{vir}$} (m-2-2);
\end{tikzpicture}
\end{equation}
commutes. Hence the surjectivity of the cyclotomic character implies 
the surjectivity of the virtual cyclotomic character 
$\chi_{vir} : \GTh_{gen} \to \wh{\ZZ}^{\times}$. 
As far as we know, this is the only way to prove the 
surjectivity of $\chi_{vir}$. See also the comment made by P. Guillot right 
before Lemma 3 in \cite{Guillot-dihedral}. 
\end{remark}

\subsubsection{The version of the Ihara homomorphism for $\GT$-shadows}
\label{sec:Ih-N}
Recall that, for every isolated object $\N$ of the groupoid $\GTSh$, the set 
$\GT(\N)$ is a finite group (in fact, $\GT(\N) = \GTSh(\N, \N)$) and the 
formula in \eqref{PR-N} defines a group homomorphism 
$\PR_{\N} :  \GTh_{gen} \to \GT(\N)$. Using the formula for 
the composition of morphisms in $\GTSh$ (see \eqref{composition}), it is easy 
to see that the equation
\begin{equation}
\label{chi-vir-N}
\chi_{vir, \N} ([m, f]) : = 2 m + 1 + N_{\ord} \ZZ
\end{equation}
defines a group homomorphism $\chi_{vir, \N} : \GT(\N) \to (\ZZ/N_{\ord} \ZZ)^{\times}$. 
Moreover, the homomorphism $\chi_{vir, \N}$ fits into the following commutative diagram: 
\begin{equation}
\label{chi-vir-chi-vir-N}
\begin{tikzpicture}
\matrix (m) [matrix of math nodes, row sep=1.5em, column sep=1.5em]
{ \GTh_{gen}  &  \GT(\N) \\
\wh{\ZZ}^{\times} &  ~(\ZZ/N_{\ord} \ZZ)^{\times} \\};
\path[->, font=\scriptsize]
(m-1-1) edge node[above] {$\PR_{\N}$} (m-1-2)
edge node[left] {$\chi_{vir}$} (m-2-1)  
(m-1-2) edge node[right] {$\chi_{vir, \N}$} (m-2-2)
(m-2-1) edge node[above] {$\hcP_{N_{\ord}}$}  (m-2-2);
\end{tikzpicture}
\end{equation}
Thus, since the homomorphisms $\chi_{vir}$ and 
$\hcP_{N_{\ord}} : \wh{\ZZ}^{\times} \to (\ZZ/N_{\ord} \ZZ)^{\times}$
are surjective, we conclude that the homomorphism $\chi_{vir, \N}$
is also surjective. 

Composing $\PR_{\N}$ with the Ihara embedding $\Ih : G_{\QQ} \into  \GTh_{gen}$ we 
get the following version of the Ihara homomorphism 
\begin{equation}
\label{Ih-N}
\Ih_{\N} : = \PR_{\N} \circ \Ih :  G_{\QQ} \to \GT(\N)
\end{equation}
for $\GT$-shadows. 

A $\GT$-shadow $[m, f] \in \GT(\N)$ that belongs to the image of $\Ih_{\N}$ is 
called\footnote{This term was suggested to us by Florian Pop.} \e{arithmetical}.
It is clear that, for every isolated object $\N$ of the groupoid $\GTSh$, the subset
of arithmetical $\GT$-shadows with the target $\N$ form a subgroup $\GT(\N)$ 
and we denote this group by $\GT_{arith}(\N)$, i.e. 
\begin{equation}
\label{GT-arith-N}
\GT_{arith}(\N) : = \Ih_{\N}(G_{\QQ}).
\end{equation}
It is also clear that every arithmetical $\GT$-shadow is genuine. 
Moreover, the existence of genuine $\GT$-shadows that are not arithmetical 
would imply that the Ihara homomorphism $\Ih : G_{\QQ} \to \GTh_{gen}$ 
is not surjective.  

Combining \eqref{chi-chi-virtual} with \eqref{chi-vir-chi-vir-N} we see that, 
for every isolated object $\N$ of $\GTSh$, the homomorphism $\Ih_{\N}$ and 
$\chi_{vir, \N}$ fit into the following commutative diagram: 
\begin{equation}
\label{Ih-N-chi-vir-N}
\begin{tikzpicture}
\matrix (m) [matrix of math nodes, row sep=1.5em, column sep=1.5em]
{ G_{\QQ}  &  \GT(\N) \\
\wh{\ZZ}^{\times} &  ~(\ZZ/N_{\ord} \ZZ)^{\times} \\};
\path[->, font=\scriptsize]
(m-1-1) edge node[above] {$\Ih_{\N}$} (m-1-2)
edge node[left] {$\chi$} (m-2-1)  
(m-1-2) edge node[right] {$\chi_{vir, \N}$} (m-2-2)
(m-2-1) edge node[above] {$\hcP_{N_{\ord}}$}  (m-2-2);
\end{tikzpicture}
\end{equation}
\begin{remark}  
\label{rem:Ih-N-for-any-N}
One may certainly consider the composition 
$$
\Ih_{\N} : =  \PR_{\N} \circ \Ih : G_{\QQ} \to \GT(\N)
$$
even if the object $\N$ is not isolated. In this case, $\GT(\N)$ does not have a natural 
group structure, so $\Ih_{\N}$ is not a group homomorphism. It does make
sense to call a $\GT$-shadow $[m, f] \in \GT(\N)$ 
\e{arithmetical} (for an arbitrary $\N \in \NFI_{\PB_3}(\B_3)$) if $[m, f]$ belongs 
to the image of $\Ih_{\N}$. We also set $\GT_{arith}(\N) := \Ih_{\N}(G_{\QQ})$
and $\GT_{arith}(\N)$ is merely a subset of $\GT(\N)$.  
\end{remark}
\begin{remark}  
\label{rem:cR-H-N-and-Ihara}
Let $\H, \N \in \NFI_{\PB_3}(\B_3)$ with $\H \subset \N$.
Since the reduction map $\cR_{\H, \N}$ satisfies 
the property $\cR_{\H, \N} \circ \PR_{\H} = \PR_{\N}$, the diagram 
$$
\begin{tikzpicture}
\matrix (m) [matrix of math nodes, row sep=1.5em, column sep=2.1em]
{G_{\QQ}  &   \GT(\H) \\
~~ & \GT(\N)   \\};
\path[->, font=\scriptsize]
(m-1-1) edge node[above] {$\Ih_{\H}$} (m-1-2)
edge node[left] {$~~~~\Ih_{\N}$} (m-2-2)  
(m-1-2) edge node[right] {$\cR_{\H, \N}$} (m-2-2);
\end{tikzpicture}
$$
commutes. Therefore, 
\begin{equation}
\label{cR-H-N-to-arithmetical}
\cR_{\H, \N} \big( \GT_{arith}(\H) \big)  =  \GT_{arith}(\N). 
\end{equation}
\end{remark}  

\subsection{The Lochak-Schneps conditions on $\GT$-shadows}
\label{sec:LS-conditions}

Using the cohomological interpretation of the defining relations of the Gro\-then\-dieck-Teich\-muel\-ler group,
P. Lochak and L. Schneps proved in \cite{LochakSchneps-CohomInt} the following remarkable statement: 
\begin{thm}[Introduction, \cite{LochakSchneps-CohomInt}]  
\label{thm:LS-conditions-hat}
Let $\te$ and $\tau$ be the automorphisms of $\wh{\F}_2$ defined in \eqref{theta} and 
\eqref{tau}, respectively. For every $(\hat{m}, \hat{f}) \in \GTh_{gen}$, there exists 
$\hat{g} \in \wh{\F}_2$ and $\hat{h} \in \wh{\F}_2$ such that 
\begin{equation}
\label{LS-conditions-hat}
\hat{f} = \te(\hat{g})^{-1} \hat{g}, \qquad 
\hat{f} x^{\hat{m}} =
\begin{cases}
\tau(\hat{h})^{-1} \hat{h}     & \textrm{if} ~~  2 \hat{m} +1  \equiv 1 \mod 3,   \\[0.18cm]
\tau(\hat{h})^{-1} x y \hat{h}     &  \textrm{if} ~~  2 \hat{m} +1  \equiv -1 \mod 3.
\end{cases} 
\end{equation}
Recall that $2\hat{m}+1$ is a unit in the ring $\Zhat$ and hence its image in 
$\ZZ/3\ZZ$ is either $\ol{1}$ or $-\ol{1}$. 
\end{thm}  
\begin{remark}  
\label{rem:LS-proof}
The statement we formulated here was proved for an arbitrary element 
of the original Grothendieck-Teichmueller group $\GTh$. However, the argument 
given on page 582 in \cite{LochakSchneps-CohomInt} works for any element of $\GTh_{gen}$. 
\end{remark} 

Theorem \ref{thm:LS-conditions-hat} gives us a useful tool for detecting 
fake $\GT$-shadows:
\begin{cor}  
\label{cor:LS-conditions}
Let $\N \in \NFI_{\PB_3}(\B_3)$ for which $3 \mid N_{\ord}$. 
If a $\GT$-shadow $[m, f] \in \GT(\N)$ is genuine, then there exist 
elements $g \in \F_2$ and  $h \in \F_2$ such that 
\begin{equation}
\label{LS-conditions-corollary}
f \,\N_{\F_2} = \te(g)^{-1}  g \, \N_{\F_2} \qquad 
f x^{m}\,\N_{\F_2} =
\begin{cases}
\tau(h)^{-1} h \, \N_{\F_2}   & \textrm{if} ~~  2 m +1  \equiv 1 \mod 3,   \\[0.18cm]
\tau(h)^{-1} x y h \, \N_{\F_2}    &  \textrm{if} ~~  2 m +1  \equiv -1 \mod 3.
\end{cases} 
\end{equation}
Equivalently, a $\GT$-shadow $[m, f] \in \GT(\N)$ is fake if the first equation 
or the second equation in \eqref{LS-conditions-corollary} cannot be solved 
for $g$ or $h$, respectively.
\end{cor}  
\begin{remark}  
\label{rem:divisible-by-3}
The technical condition in Corollary \ref{cor:LS-conditions}, $3 \mid N_{\ord}$, is not very 
restrictive: one can show that for every $\N \in \NFI_{\PB_3}(\B_3)$ there exists  
$\K \in \NFI_{\PB_3}(\B_3)$ such that $\K \le \N$ and $3 \mid K_{\ord}$. 
If $3 \nmid N_{\ord}$ and $[m, f] \in \GT(\N)$, then one may still try to use 
Corollary \ref{cor:LS-conditions} to test whether $[m, f]$ is fake 
as follows: 

\begin{enumerate}

\item if $[m, f]$ does not belong to the image of the reduction map 
$\cR_{\K, \N} : \GT(\K) \to \GT(\N)$, then $[m, f]$ is fake (see \cite[Corollary 5.4]{GTgentle});

\item if the first equation or the second equation in \eqref{LS-conditions-corollary} cannot be 
solved for every $\GT$-shadow in $\cR_{\K, \N}^{-1} ([m, f]) \subset \GT(\K)$, then 
the $\GT$-shadow $[m, f]$ is fake. 

\end{enumerate}
\end{remark}
Let $[m, f] \in \GT(\N)$ with $3 \mid N_{\ord}$. If there exist 
elements $g \in \F_2$ and $h \in \F_2$ such that 
\eqref{LS-conditions-corollary} holds, then we say that
the $\GT$-shadow $[m, f] \in \GT(\N)$ satisfies the \e{Lochak-Schneps conditions}.

\begin{remark}  
\label{rem:complex-conjugation}
Due to \cite[Theorem 1]{LochakSchneps-CohomInt}, the image of the complex conjugation
with respect to the Ihara embedding is the pair\footnote{It should be mentioned that P. Lochak and L. Schneps 
use $2 \hat{m} + 1$ instead of $\hat{m}$. Luckily, $2(-1) + 1 = -1$.} $(-1_{\Zhat},1_{\wh{\F}_2})$. Thus the pairs 
$(0,1)$ and $(-1,1)$ represent arithmetical (and hence genuine) $\GT$-shadows with a target $\N$ for 
every $\N \in \NFI_{\PB_3}(\B_3)$. The pair $(0, 1)$ represents the $\GT$-shadow $\Ih_{\N}(1_{G_{\QQ}}) \in \GT(\N)$
and the pair $(-1, 1)$ represents the image of the complex conjugation with respect to the map
$\Ih_{\N} : G_{\QQ} \to \GT(\N)$.
\end{remark}

\subsection{Organization of the paper}
\label{sec:organization}
In Section \ref{sec:GTSh}, we give a brief reminder of the groupoid $\GTSh$ of $\GT$-shadows \cite{GTgentle}
for the gentle version $\GTh_{gen}$ of the Grothendieck-Teichmueller group. 
In Section \ref{sec:Dih}, we consider the family of group homomorphisms 
$\psi_n : \PB_3 \to D_n \times D_n \times D_n$ defined by the formulas 
$$
\psi_n(x_{12}) := (r, s, s), \qquad 
\psi_n(x_{23}) := (rs, r, rs), \qquad \psi_n(c) := (1,1,1),
$$
where $D_n$ is the dihedral group $\lan r , s ~|~ r^n, ~s^2,~ rsrs^{-1} \ran$ of order $2 n$ and 
$n \in \ZZ_{\ge 3}$. Due to Proposition \ref{prop:K-n-in-NFI}, the subgroup $\K^{(n)} := \ker(\psi_n)$
is an element of the poset $\NFI_{\PB_3}(\B_3)$. So we set 
$$
\Dih := \{\K^{(n)} ~|~ n \in \ZZ_{\ge 3} \}
$$
and call $\Dih$ the \e{dihedral poset}. 

In Section \ref{sec:GT-Kn}, we describe 
the set $\GT(\K)$ for an arbitrary element $\K \in \Dih$. This description is
presented in Theorem \ref{thm:GT-Kn-set}. Due to the same theorem, every $\K \in \Dih$ 
is the only object of its connected component in the groupoid $\GTSh$.
Therefore, for every $\K \in \Dih$, $\GT(\K) = \GTSh(\K, \K)$ and $\GT(\K)$ is a (finite) group.
In this section, we also prove that the reduction homomorphism 
$\cR_{\H, \K}: \GT(\H) \to \GT(\K)$ is surjective for every pair $\H, \K \in \Dih$ with $\H \le \K$ 
(see Theorem \ref{thm:GTKn-redmap-strong}) and give an explicit description of the (finite) group 
$\GT(\K)$ for every $\K \in \Dih$ (see Proposition \ref{prop:pair-multiplication} and 
Theorem \ref{thm:GT-Kn-group}).  

Section \ref{sec:arithmetical} contains the main result of this paper (see Theorem \ref{thm:main}).
First, we prove that all $\GT$-shadows with the target $\K \in \Dih$ satisfy the Lochak-Schneps conditions
(see Theorem \ref{thm:LS-conditions}). Second, Theorem \ref{thm:main} gives an explicit bound on the number 
of arithmetical $\GT$-shadows with the target $\K \in \Dih$. We use this bound to prove that, for every $\al \in \ZZ_{\ge 2}$, 
the version of the Ihara homomorphism $\Ih_{\K^{(2^{\al})} } : G_{\QQ} \to \GT(\K^{(2^{\al})})$ is surjective. 
Finally, we construct an explicit example of a non-abelian (infinite) profinite 
quotient of $\GTh$ that receives the surjective homomorphism from $G_{\QQ}$ 
(see Theorem \ref{thm:infinite-quotient} and Proposition \ref{prop:lim-ML-Dih2}).

Appendix \ref{app:commutator-subgroup} contains the proof of the technical proposition 
about the commutator subgroup $[\F_2/\K_{\F_2} , \F_2/\K_{\F_2}]$ of 
$\F_2/\K_{\F_2}$ for $\K \in \Dih$.  

\subsection{Comments about the existing literature}
\label{sec:existing-lit}
The results of papers \cite{Coleman-Anderson-Ihara}, \cite{Ichimura-Kaneko} and \cite{Ihara1986} 
are probably the closest to the present\footnote{We are thankful to the anonymous referee for 
mentioning to us two very important ``pre-$\GTh$'' papers \cite{Ichimura-Kaneko}, \cite{Ihara1986}.} 
article. 

Let $l$ be a prime integer. In \cite[Parts II, III]{Ihara1986}, Y. Ihara uses certain automorphisms 
of the pro-$l$ completion of $\F_2$ and the ring  $\ZZ_l[[u, v]]$ of formal Taylor 
power series to construct an infinite non-abelian profinite group $\Psi^-$
and a group homomorphism 
\begin{equation}
\label{psi-QQ}
\psi_{\QQ} : G_{\QQ} \to \Psi^{-}\,.
\end{equation}
See the Corollary after Theorem 8 on page 84 of \cite{Ihara1986}.

Let $\QQ(\mu_{l^{\infty}})$ be the field obtained from $\QQ$ by adjoining all roots of $1$ of 
orders $l^t$, $t \ge 1$. Let us also assume that $l$ is an odd prime.  
In paper \cite{Ichimura-Kaneko}, H. Ichimura and M. Kaneko construct
a group homomorphism $\IK$ from the absolute Galois group $G_{\QQ(\mu_{l^{\infty}})}$ of 
$\QQ(\mu_{l^{\infty}})$ to the subgroup $\Psi^{-}_1 \le \Psi^-$ (in paper  \cite{Ichimura-Kaneko},  
the group $\Psi^{-}_1$ is denoted by $\mF$). In paper \cite{Ichimura-Kaneko}, the authors use 
Coleman's Theorem \cite[Theorem 7.3]{Coleman-Anderson-Ihara} to prove 
that the image of $\IK$ coincides with the subgroup $\Psi^-_1$ if and only if 
the Kummer-Vandiver conjecture \cite[Section 8.3]{Washington_Cyclotomic} holds for $l$. 
(At the time of writing, the Kummer-Vandiver conjecture is verified \cite{Vandiver-check} for all 
odd primes $l < 2^{31}$.) Combining these statements with the surjectivity of the $l$-adic cyclotomic character 
$G_{\QQ} \to \ZZ_{l}^{\times}$, we conclude that the homomorphism in \eqref{psi-QQ} is surjective for every 
odd prime $l< 2^{31}$. 

Using the constructions developed in \cite[Section 4]{Drinfeld}, \cite{Ihara-ICM}, \cite{Ihara}, it is not hard to see that the 
homomorphism in \eqref{psi-QQ} is the composition of the Ihara embedding $\Ih : G_{\QQ} \to \GTh$ with the 
natural homomorphisms $\GTh \to \GT_l$ and $\GT_l \to \Psi^{-}$. 

Thus, for every odd prime $l < 2^{31}$, $\Psi^{-}$ is an infinite (non-abelian) profinite quotient of $\GTh$, 
and $\Psi^{-}$ receives a surjective homomorphism from $G_{\QQ}$. Thanks to R. Sharifi's note \cite{Sharifi}, a similar 
statement can be proved for $l = 2$.

For a prime $p$, we denote by $\NFI^{p}_{\PB_3}(\B_3)$ the following subposet of 
$\NFI^{isolated}_{\PB_3}(\B_3)$
$$
\NFI^{p}_{\PB_3}(\B_3) : = \{\N \in \NFI^{isolated}_{\PB_3}(\B_3) ~|~ \PB_3/\N  \txt{ is a }p\txt{-group}\}. 
$$
In paper \cite{Andre}, Y. Andr\'e introduces a $p$-adic version $\GT^{temp}_p$ of $\GTh$ and connects it to 
the absolute Galois group $\Gal(\ol{\QQ}_p/\QQ_p)$ of $\QQ_p$. Since  Andr\'e's definition 
of $\GT^{temp}_p$ involves the tempered fundamental group of a rigid-analytic manifold 
obtained from $\PP^1_{\QQ} - \{0, 1, \infty\}$, and there is no obvious link between 
$\GT^{temp}_p$ and 
$$
\lim\big( \ML \big|_{\NFI^{p}_{\PB_3}(\B_3)} \big),
$$
we do not see how the main result of \cite{Andre} (in the case $p = 2$) can be connected to the results of our paper.  

In paper \cite{CombeKalugin}, N.C. Combe and A. Kalugin considered the dihedral symmetry 
relation in the fundamental groupoids of the configuration spaces of points in $\CC$ and 
they proved that the action of the pro-unipotent version $\GTh^{un}$ of the Grothendieck-Teichmueller group 
preserves this dihedral symmetry. We do not see any link between the dihedral symmetry 
relation considered in \cite{CombeKalugin} and the dihedral poset $\Dih$ introduced in this paper.

In papers \cite{Guillot0} and \cite{Guillot-dihedral}, P. Guillot
investigated a similar construction related to the group $\GTh_{gen}$. 
He used an equivalent but quite different definition of $\GTh_{gen}$ 
(see \cite[Main Theorem, (a)]{Harbater-Schneps}) to get the ``set-up of finite quotients''. 
For this reason, it is not easy to compare the groupoid $\GTSh$ to Guillot's construction.
In \cite[Section 3]{Guillot-dihedral}, the author computes the kernels of his version 
of the virtual cyclotomic character for a family of characteristic subgroups of $\F_2$
related to the dihedral groups $D_n$. However, as far as we understand, the author only
considers the cases when $n = 2^{\al} n_0$, where $n_0$ is odd and $\al \in \{0, 1, 2\}$.

\subsection{Notational conventions}
For a set $X$ with an equivalence relation and $a \in X$ we will denote by $[a]$
the equivalence class which contains the element $a$. The notation $\gcd$ (resp. $\lcm$) 
is reserved for the greatest common divisor (resp. the least common multiple). 
The notation $\phi$ is reserved for Euler's phi function (a.k.a. the totient function).   
For a group $G$, $\wh{G}$ denotes the profinite completion of $G$. Every 
finite group is considered as the topological group (with the discrete topology).  

The notation $\B_n$ (resp. $\PB_n$) is reserved for the Artin braid group 
on $n$ strands (resp. the pure braid group on $n$ strands). $S_n$ denotes the 
symmetric group on $n$ letters. We denote by $\si_1$ and $\si_2$ the standard 
generators of $\B_3$. Furthermore, $x_{12}$, $x_{23}$ and $c$ are
the following generators of $\PB_3$:
$$
x_{12} := \si_1^2, \qquad x_{23}:=\si_2^2, \qquad  c := (\si_1 \si_2 \si_1)^2\,.
$$
We recall that the element  $c := (\si_1 \si_2 \si_1)^2$ belongs to the center $\cZ(\PB_3)$ of 
$\PB_3$ (and the center $\cZ(\B_3)$ of $\B_3$). Moreover, $\cZ(\B_3) = \cZ(\PB_3) = \lan c \ran \cong \ZZ$. 

We set $\D :=  \si_1 \si_2 \si_1$ and observe that 
\begin{equation}
\label{si-D-si}
\si_1 \D = \D \si_2, \qquad \si_2 \D = \D \si_1, 
\qquad 
\si_1^{-1} \D = \D \si_2^{-1}, 
\qquad 
\si_2^{-1} \D = \D \si_1^{-1},
\end{equation}
\begin{equation}
\label{Delta-2-c}
\D^2 = c.
\end{equation}

Using identities \eqref{si-D-si} and \eqref{Delta-2-c}, it is easy to see that the adjoint 
action of $\B_3$ on $\PB_3$ is given on generators by the formulas: 
\begin{equation}
\label{conj-by-si1}
\si_1 x_{12} \si_1^{-1}  = \si_1^{-1} x_{12} \si_1 = x_{12}, \qquad 
\si_1 x_{23} \si_1^{-1} = x_{23}^{-1} x_{12}^{-1} c, \qquad 
\si_1^{-1} x_{23} \si_1 = x_{12}^{-1} x_{23}^{-1}  c,
\end{equation} 
\begin{equation}
\label{conj-by-si2}
\si_2 x_{12} \si_2^{-1} = x^{-1}_{12} x^{-1}_{23} c, \qquad 
\si_2^{-1} x_{12} \si_2 = x_{23}^{-1} x_{12}^{-1} c \qquad
\si_2 x_{23} \si_2^{-1} = \si_2^{-1} x_{23} \si_2 =  x_{23}\,.
\end{equation}  
Moreover, 
\begin{equation}
\label{conj-by-D}
\D x_{12} \D^{-1} = x_{23}, \qquad 
\D x_{23} \D^{-1} = x_{12}\,.
\end{equation}

It is known \cite[Section 1.3]{Braids} that $\lan x_{12}, x_{23} \ran$ is isomorphic to the free group $\F_2$ on two 
generators, and $\PB_3  \cong \lan x_{12}, x_{23} \ran \times \lan c \ran$. In this paper, we tacitly identify $\F_2$ with the subgroup 
$\lan x_{12}, x_{23} \ran \le \PB_3$ and we often use the following 
notation for $x_{12}$, $x_{23}$ and $(x_{12} x_{23})^{-1}$:
$$
x := x_{12}, \qquad y := x_{23}, \qquad  z := y^{-1} x^{-1}\,.
$$

We denote by $\te$ and $\tau$ the automorphisms of $\F_2:= \lan x,y \ran$
defined by the formulas
\begin{equation}
\label{theta}
\te(x):= y, \qquad \te(y):= x,
\end{equation}
\begin{equation}
\label{tau}
\tau(x):= y, \qquad \tau(y):= y^{-1} x^{-1}\,.
\end{equation}
We use the same letters $\te$ and $\tau$ for the corresponding continuous 
automorphisms of $\wh{\F}_2$ (see \cite[Corollary A.2]{GTgentle}). 

For a group $G$, $\End(G)$ is the monoid of endomorphisms $G \to G$ and
the notation $[G, G]$ is reserved for the commutator subgroup of $G$. 
For a subgroup $H\le G$, the notation $|G:H|$ is reserved for the index of $H$ in $G$. 
For a finite group $G$, $|G|$ denotes the order of $G$.
For a normal subgroup $H\unlhd G$ of finite index, we denote by $\NFI_{H}(G)$ the poset 
of finite index normal subgroups $\N$ in $G$ such that $\N \le H$. Moreover, 
$\NFI(G) := \NFI_G(G)$, i.e. $\NFI(G)$ is the poset of normal finite index subgroups of a group $G$.
For a subgroup $H\le G$, $\Core_G(H)$ denotes the normal core of $H$ in $G$, i.e. 
$$
\Core_G(H) : = \bigcap_{g \in G} g H g^{-1}\,.
$$ 

For $\N \in \NFI_{\PB_3}(\B_3)$ we set 
$$
N_{\ord} := \lcm(\ord(x_{12}\N),\ord(x_{23}\N), \ord(c\N)), \qquad 
\N_{\F_2}:=\N\cap \F_2\,.
$$
We denote by $\GT_{pr}(\N)$ the set of $\GT$-pairs with the target $\N$, by $\GT_{pr}^\hs(\N)$ the set of charming $\GT$-pairs 
with the target $\N$ and by $\GT(\N)$ the set of $\GT$-shadows with the target $\N$ (see Definitions \ref{dfn:GT-pair} and
\ref{dfn:GT-shadow}.) For a pair $(m, f) \in \ZZ \times \F_2$, we denote by $[m, f]$ the corresponding 
element in $\ZZ / N_{\ord}\ZZ \times \F_2/\N_{\F_2}$, i.e. 
$[m, f] : = \big( m +  N_{\ord}\ZZ, f \, \N_{\F_2} \big)$. Note that we slightly abuse 
the notation: the positive integer $N_{\ord}$ depends on the choice of the finite 
index normal subgroup $\N \unlhd \PB_3$. Later, we introduce a family of finite index normal 
subgroups $(\K^{(n)})_{n \ge 3}$ of $\PB_3$. For each $n \ge 3$, the integer $K^{(n)}_{\ord}$ 
is the least common multiple of the orders of the cosets of $x_{12}$, $x_{23}$ and $c$ in 
the finite group $\PB_3/\K^{(n)}$.

For a commutative ring $R$, the notation $R^{\times}$ is reserved for the (multiplicative) group 
of units of $R$. We denote by $\Aff(R)$ the group of affine transformations of $R$, i.e. 
$$
\Aff(R) := R \rtimes R^{\times}
$$
with the obvious action of $R^{\times}$ on the underlying group $(R, +)$: 
$(u, r) \mapsto u \, r : R^{\times} \times R \to R$. 

Given $n, q\in \ZZ_{\ge 1}$ with $n\mid q$ and residue classes $\ol{k}_1, \ol{k}_2 \in \ZZ/ q\ZZ$, we 
occasionally write $k_1 \equiv k_2 \mod n$ meaning that the images of $\ol{k}_1, \ol{k}_2$ in 
$\ZZ/ n\ZZ$ coincide.

Given $n, q\in \ZZ_{\ge 1}$ with $n\mid q$ we also consider the following group
\begin{equation}
\label{Zn-rtimes-units-Zq}
\ZZ/n\ZZ \rtimes (\ZZ/q\ZZ)^{\times},
\end{equation}
where the action of $(\ZZ/q\ZZ)^{\times}$ on $(\ZZ/n\ZZ, +)$ comes 
from the natural surjective group homomorphism $(\ZZ/q\ZZ)^{\times} \to (\ZZ/n\ZZ)^{\times}$ 
and the above action of $(\ZZ/n\ZZ)^{\times}$ on $(\ZZ/n\ZZ, +)$. 
It is obvious that, for every $n \in \ZZ_{\ge 1}$, $\ZZ/n\ZZ \rtimes (\ZZ/n\ZZ)^{\times} = \Aff(\ZZ/n\ZZ)$.\\

The cyclotomic character $G_{\QQ} \to \Zhat^{\times}$ is denoted by $\chi$. 


\bigskip
\noindent
{\bf Acknowledgements.} We are thankful to Ihor Pylaiev for his support, for his contribution to the project and for working 
with us in the Yulia's Dream program. 
We are thankful to Pavel Etingof, 
Slava Gerovitch, and Dmytro Matvieievskyi for giving us the opportunity to participate in Yulia's Dream program. 
We are thankful to Jacob Guynee for his help with presenting selected proofs of this paper. 
We are thankful to Leila Schneps for her suggestion that prompted us to formulate the main 
result of this paper differently. We are thankful to Frauke Bleher, Ted Chinburg, Benjamin Collas, 
Benjamin Enriquez, Pavel Etingof, Pierre Guillot, David Harbater, Jaclyn Lang, 
Florian Pop, Aniruddha Sudarshan, Dmitry Tamarkin, 
John Voight, Alexander Voronov for their suggestions and 
discussions.  I.B., B.H. and V.P. are thankful to their parents for their help and support. 
I.B. and V.P. are thankful to the United Kingdom for its hospitality. B.H. acknowledges 
the Ukraine Support Program of Bar-Ilan University for their generous support. 
V.A.D. is thankful to Thomas Willwacher who suggested him (during a walk in October of 2016) 
to find an {\it interesting} family of objects of the groupoid of $\GT$-shadows.    
V.A.D. acknowledges a partial support from the project ``Arithmetic and Homotopic Galois Theory''. This project is, in turn,  
supported by the CNRS France-Japan AHGT International Research Network between the RIMS Kyoto University, 
the LPP of Lille University, and the DMA of ENS PSL. V.A.D. also acknowledges Temple University 
for 2021 Summer Research Award. We would like to thank the anonymous referee for her/his very useful feedback and 
for mentioning to us about paper \cite{Ichimura-Kaneko} by H. Ichimura and M. Kaneko.

\section{Reminder of the groupoid GTSh}
\label{sec:GTSh}
In this section, we list all the necessary definitions and statements. The statements are presented without proofs, 
as those can be found in \cite[Section 3]{GTgentle}. 
Recall that, for $\N \in \NFI_{\PB_3}(\B_3)$, $N_{\ord} := \lcm(\ord(x_{12}\N),\ord(x_{23}\N), \ord(c\N))$ 
and $\N_{\F_2}:=\N \cap \F_2$.
\begin{defi}
\label{dfn:GT-pair}
A $\GT$-\textbf{pair with the target} $\N$ is a pair 
$$
(m + N_\ord \ZZ,f\N_{\F_2})\in \ZZ / N_\ord \ZZ \times \F_2/\N_{\F_2}
$$ 
satisfying the {\bf hexagon relations} modulo $\N$:
\begin{equation}
\label{eq:fhex}
\sigma_1^{2m+1}f^{-1}\sigma_2^{2m+1}f\, \N=f^{-1}\sigma_1\sigma_2x_{12}^{-m}c^m\, \N
\end{equation}
and
\begin{equation}
\label{eq:shex}
f^{-1}\sigma_2^{2m+1}f\sigma_1^{2m+1}\, \N=\sigma_2\sigma_1x_{23}^{-m}c^m f\,\N.
\end{equation}
If, in addition, the $\GT$-pair satisfies
$$
\gcd(2m+1, N_\ord) = 1  \quad \text{and} \quad f\N_{\F_2} \in [\F_2/\N_{\F_2},\F_2/\N_{\F_2}],
$$
then we call it \textbf{charming}. We denote by $\GT_{pr}(\N)$ the set of $\GT$-pairs with 
the target $\N$ and by $\GT_{pr}^\hs(\N)$ the set of charming $\GT$-pairs 
with the target $\N$.
\end{defi}
It is easy to see from the definitions of $N_\ord$ and $\N_{\F_2}$ that, if a pair 
$(m, f) \in \ZZ \times \F_2$ satisfies \eqref{eq:fhex}, \eqref{eq:shex} then all elements of the coset 
$(m + N_\ord\mathbb{Z}, f\N_{\F_2})$ satisfy \eqref{eq:fhex}, \eqref{eq:shex}.

If a pair $(m, f) \in  \ZZ \times \F_2$ satisfies an additional condition, then,
due to the following proposition, relations  \eqref{eq:fhex}, \eqref{eq:shex}
are equivalent to the \textbf{simplified hexagon relations} (see \eqref{s-hex1} and 
\eqref{s-hex11}).
\begin{prop}
\label{prop:simple-hexa}
Let $\N \in \NFI_{\PB_3}(\B_3)$ and $\theta$ and $\tau$ be the automorphisms of $\F_2$ defined 
in \eqref{theta} and \eqref{tau}, respectively. 
A pair $(m,f) \in \mathbb{Z}\times [\F_2,\F_2]$ satisfies the hexagon relations modulo $\N$ if and only if
\begin{equation}
\label{s-hex1}
f \te(f) \in \N_{\F_2}
\end{equation}
and
\begin{equation}
\label{s-hex11}
\tau^2(y^m f)\tau(y^m f)y^m f \in \N_{\F_2}.
\end{equation}
\end{prop}

Due to the hexagon relations, the map $T_{m, f}$ defined by 
\begin{equation}
T_{m,f}(\sigma_1)=\sigma_1^{2m+1}\, \N, \qquad T_{m,f}(\sigma_2)=f^{-1}\sigma_2^{2m+1}f\, \N
\end{equation}
is a group homomorphism from $\B_3$ to $\B_3/\N$. 
The homomorphisms $T_{m,f}^{\PB_3} : \PB_3\rightarrow \PB_3/\N$ and 
$T_{m,f}^{\F_2} : \F_2 \rightarrow \F_2/\N_{\F_2}$ are the restrictions of 
the homomorphism $T_{m,f}$ to $\PB_3$ and $\F_2$, respectively.
Notice that 
$$
\ker(T_{m,f})=\ker(T_{m,f}^{\PB_3}) \in \NFI_{\PB_3}(\B_3).
$$

From now on, $[m, f]$ denotes the $\GT$-pair $(m + N_\ord\mathbb{Z}, f\N_{\F_2})$
represented by $(m, f) \in \ZZ \times \F_2$. Recall that, for every 
charming $\GT$-pair $[m, f]$ with the target $\N$, the surjectivity of $T_{m,f}$ is equivalent 
to that of $T_{m,f}^{\PB_3}$ and to that of $T_{m,f}^{\F_2}$. 

We can now define a $\GT$-shadow with the target $\N$.
\begin{defi}
\label{dfn:GT-shadow}
A charming $\GT$-pair $[m, f]$ is called a $\GT$-\e{shadow with the target} $\N$ 
if $T_{m,f}$ is surjective. The set of $\GT$-shadows with the target $\N$ is denoted by $\GT(\N)$.
\end{defi}
Since for every $[m, f] \in \GT(\N)$, the group homomorphisms 
$T_{m, f} : \B_3 \to \B_3/\N$, $T^{\PB_3}_{m, f} : \PB_3 \to \PB_3/\N$ and 
$T^{\F_2}_{m, f} : \F_2 \to \F_2/\N_{\F_2}$ are onto, $T_{m, f}$, $T^{\PB_3}_{m, f}$, 
and $T^{\F_2}_{m, f}$ induce the isomorphisms 
$$
T^{\isom}_{m, f} :  \B_3/\K \iso \B_3/\N, \quad
T^{\PB_3, \isom}_{m, f} :  \B_3/\K \iso \B_3/\N, \quad
T^{\F_2, \isom}_{m, f} :  \F_2/\K_{\F_2} \iso \F_2/\N_{\F_2},
$$
respectively, where $\K := \ker(T_{m, f})$.

This observation implies the first three statements of the following 
proposition\footnote{See also \cite[Proposition 3.8]{GTgentle}.}: 
\begin{prop}
\label{prop:quotients-isomorphic}
Let $\K, \N \in \NFI_{\PB_3}(\B_3)$. If there exists $[m,f] \in \GT(\N)$ such that $\K = \ker(T_{m,f})$ then
\begin{enumerate}

\item the finite groups $\B_3/\N$ and $\B_3/\K$ are isomorphic (and hence $|\B_3:\N|=|\B_3:\K|$);

\item the finite groups $\PB_3/\N$ and $\PB_3/\K$ are isomorphic (and hence $|\PB_3:\N|=|\PB_3:\K|$);

\item the finite groups $\F_2/\N_{\F_2}$ and $\F_2/\K_{\F_2}$ are isomorphic 
(and hence $|\F_2: \N_{\F_2}| = |\F_2:\K_{\F_2}|$);

\item $K_\ord = N_\ord$.

\end{enumerate}
\end{prop}
Just as in \cite[Section 3]{GTgentle}, for $(m,f) \in \ZZ \times \F_2$, we 
denote by $E_{m, f}$ the following endomorphism of $\F_2$:
$$
E_{m,f}(x)=x^{2m+1}, \quad E_{m,f}(y)=f^{-1}y^{2m+1}f.
$$
We recall that $\GT$-shadows form a groupoid $\GTSh$ with 
$\Ob(\GTSh):=\NFI_{\PB_3}(\B_3)$ and 
$$
\GTSh(\K,\N):=\{[m,f] \in \GT(\N)~|~ \ker(T_{m,f}) = \K\}.
$$
The following theorem describes the composition of morphisms in $\GTSh$:
\begin{thm}
\label{thm:GTSh}
Let $\N^{(1)}, \N^{(2)}, \N^{(3)} \in \NFI_{\PB_3}(\B_3)$,
$$
[m_1, f_1] \in \GTSh(\N^{(2)}, \N^{(1)}), \qquad [m_2, f_2] \in \GTSh(\N^{(3)},\N^{(2)})
$$
and  $N_{\ord} := N^{(1)}_{\ord} = N^{(2)}_{\ord} = N^{(3)}_{\ord}$. Then
\begin{enumerate}
\item the formula
\begin{equation}
\label{composition}
[m_1,f_1] \circ [m_2,f_2] = [2m_1m_2+m_1+m_2,f_1E_{m_1,f_1}(f_2)]
\end{equation}
defines the composition of morphisms in $\GTSh$; 

\item for every $\N \in \NFI_{\PB_3}(\B_3)$, the pair $(0,1_{\F_2})$ represents the identity morphism in $\GTSh(\N,\N)$;

\item finally, for every $[m, f] \in \GTSh(\K,\N)$, the formulas
$$
\tilde{m} + N_\ord\mathbb{Z} := 
-(2\overline{m}+1)^{-1}\overline{m},\quad \tilde{f}\K_{\F_2}:=(T_{m,f}^{\F_2, \isom})^{-1}(f^{-1}\N_{\F_2})
$$
define the inverse $[\tilde{m},\tilde{f}]\in\GTSh(\N,\K)$ of the morphism $[m,f]$.
\end{enumerate}
\end{thm}
For the proof, see \cite[Theorem 3.10]{GTgentle}.

\section{The dihedral poset}
\label{sec:Dih}
Let $n \in \ZZ_{\ge 3}$ and $D_n := \lan r, s ~|~ r^n ,~s^2,~ srs^{-1}r  \ran$. 
We start with the group homomorphism $\psi_n : \PB_3 \to D_n^3$ defined by the formulas:
\begin{equation}
\label{psi-n}
\psi_n(x_{12}) := (r, s, s), \qquad 
\psi_n(x_{23}) := (rs, r, rs), \qquad \psi_n(c) := (1,1,1).
\end{equation}

We set $\K^{(n)}: = \ker(\psi_n)$ and claim that
\begin{prop}
\label{prop:K-n-in-NFI}
For every $n \in \ZZ_{\ge 3}$, $\K^{(n)}$ belongs to the poset $\NFI_{\PB_3}(\B_3)$.
\end{prop}
\begin{proof}
First, note that $\K^{(n)}$ is a finite index subgroup of $\PB_3$ because $D_n^3$ is finite. 
The subgroup $\PB_3$ also has finite index in $\B_3$, so $\K^{(n)}$ has finite index in $\B_3$. 
Thus it remains to show that $\K^{(n)}$ is normal in $\B_3$.

Consider the map $\vf : \PB_3 \to D_n$ given by
\begin{equation*}
\vf(x_{12}) := s, \qquad \vf(x_{23}) := rs, \qquad \vf(c) :=1. 
\end{equation*}
We will show that $\K^{(n)}$ is the normal core in $\B_3$ of $\ker(\vf) \leq \PB_3$. 
Define for $w \in \B_3$ the map $\vf^w : \PB_3 \to D_n$ given by
\begin{equation*}
\vf^w(g) := \vf(w^{-1} g w), \qquad g \in \PB_3.
\end{equation*}
Note that
\begin{equation*}
\ker(\vf^w) = w \ker(\vf) w^{-1},
\end{equation*}
and hence
\begin{equation*}
C := \Core_{\B_3}\big(\ker(\vf) \big) = \bigcap_{w \in \B_3} \ker(\vf^w). 
\end{equation*}
Since $|\B_3 : \PB_3|  = 6$ and that the elements
\begin{equation*}
1, ~~ \si_1^{-1}, ~~\si_2^{-1}, ~~\D^{-1},~~ 
\si_1^{-1} \D^{-1},~~ \si_2^{-1} \D^{-1}
\end{equation*}
form a complete set of coset representatives, we have 
\begin{equation}
\label{Core-better}
C = \ker(\vf) \cap \ker (\vf^{\si_1^{-1}}) \cap 
\ker (\vf^{\si_2^{-1}}) \cap \ker (\vf^{\D^{-1}}) \cap
\ker(\vf^{ \si_1^{-1} \D^{-1}}) \cap \ker(\vf^{ \si_2^{-1} \D^{-1}}).
\end{equation}
We will now show that 
\begin{equation}
\label{normal-core-intersection-of-three}
 C = \ker(\vf) \cap \ker(\vf^{\si_1^{-1}}) \cap \ker(\vf^{\si_2^{-1}}).
\end{equation}

Let $\ga$ be the following automorphism of $D_n$
$$
\ga(r) := r^{-1}, \qquad \ga(s) := r s.
$$
Clearly, $\ga(s) = rs$ and $\ga(rs) = s$ or equivalently 
$\ga \circ \vf(x_{12}) = \vf(x_{23})$ and $\ga \circ \vf(x_{23}) = \vf(x_{12})$.
Since conjugation by $\D$ swaps $x_{12}$ with $x_{23}$, and $\vf(c) = 1$, we have
$$
\vf^{w\D^{-1}}(g) = \ga \circ \vf^w(g), \quad \forall~ g \in \PB_3, ~ w \in \B_3\,.
$$
This gives us 
\begin{equation*}
\ker(\vf) = \ker(\vf^{\D^{-1}}), \qquad \ker(\vf^{\si_1^{-1}}) = \ker(\vf^{\si_1^{-1}\D^{-1}}), \qquad 
\ker(\vf^{\si_2^{-1}}) = \ker(\vf^{\si_2^{-1}\D^{-1}}),
\end{equation*}
which proves \eqref{normal-core-intersection-of-three}. 

Let $\ti{\psi} := \vf^{\si_{2}^{-1}} \times \vf^{\si_1^{-1}} \times \vf : \PB_3 \to D_n^3$.  
Using \eqref{conj-by-si1} and \eqref{conj-by-si2} we see that 
$$
\ti{\psi}(x_{12}) = (r^{-1}, s, s), 
\qquad 
\ti{\psi}(x_{23}) = (rs, r, rs), 
\qquad 
\ti{\psi}(c) = (1,1,1).
$$
Identity \eqref{normal-core-intersection-of-three} implies that $C = \ker(\ti{\psi})$. 

Let $j$ be the following inner automorphism of $D_n^3$:
$$
j(g_1,g_2, g_3) : = (rs (g_1) (rs)^{-1}, g_2, g_3).
$$
Since $\psi_n = j \circ \ti{\psi}$, we have
$$
\K^{(n)} := \ker(\psi_n) =  \ker(\ti{\psi}) = C. 
$$ 
Since $\K^{(n)} \unlhd \B_3$, we proved that $\K^{(n)} \in \NFI_{\PB_3}(\B_3)$.
\end{proof}

We denote by $\Dih$ the subposet of $\NFI_{\PB_3}(\B_3)$
$$
\Dih := \{\K^{(n)} ~|~ n \in \ZZ_{\ge 3} \}  
$$
and call it the \e{dihedral poset}.

\begin{remark}  
\label{rem:Kn-F2-Kn-ord}
For every $n \in \ZZ_{\ge 3}$, $\K^{(n)}_{\F_2}$ is the kernel of the homomorphism $\F_2 \to D_n^3$ that 
sends $x$ to $(r, s, s)$ and $y$ to $(rs, r, rs)$. Moreover, 
\begin{equation}
\label{Kn-ord}
K^{(n)}_{\ord} = \lcm(n, 2).
\end{equation}
\end{remark}  
\begin{remark}  
\label{rem:Kq-Kn}
If $q, n \in \ZZ_{\ge 3}$, $n ~|~q$, and 
$D_q = \lan a, b~|~ a^q ,~b^2,~ bab^{-1}a  \ran$, then 
the formulas
\begin{equation}
\label{Dq-to-Dn}
\eta_{q,n}(a) := r, \qquad 
\eta_{q,n}(b) := s
\end{equation}
define a natural homomorphism $\eta_{q,n} : D_q \to D_n$. 
Since $\eta^3_{q, n} \circ \psi_q = \psi_n$, we have 
$\K^{(q)} \le \K^{(n)}$. 
\end{remark}

It is convenient to identify $\F_2/ \K^{(n)}_{\F_2}$ with the subgroup
$$
G_n : = \lan (r, s, s),   (rs, r, rs) \ran \le D_n^3\,.
$$ 
For $w \in \F_2$, $\ol{w}$ denotes the coset $w \K^{(n)}_{\F_2}$. Thus, 
\begin{equation}
\label{xyz-bar}
\ol{x} = (r, s, s), \qquad \ol{y} = (rs, r, rs), 
\qquad 
\ol{z} = (r^2 s, r^{-1} s, r).
\end{equation}

We will now use the identification $\F_2/ \K^{(n)}_{\F_2} \cong G_n$ 
to prove two important propositions about the structure of the dihedral poset $\Dih$:
\begin{prop}
\label{prop:K-2n-n}
For every odd integer $n \ge 3$,  we have $\K^{(n)} = \K^{(2n)}$.
\end{prop}
\begin{proof}
Due to Remark \ref{rem:Kq-Kn}, $\K^{(2n)} \subset \K^{(n)}$. 

As above, we identify the quotient group $\F_2/\K_{\F_2}^{(q)}$ with the 
subgroup $G_q$ of $D_q^3$ generated by $\ol{x}$ and $\ol{y}$.
Due to \eqref{xyz-bar}, 
$$
\ol{x}^2 = (r^2, 1, 1), \qquad 
\ol{y}^2 = (1, r^2, 1), \qquad 
\ol{z}^2 = (1, 1, r^2).
$$
Hence $J_q := \langle r^2 \rangle \times \langle r^2 \rangle \times \langle r^2 \rangle $ is a subgroup of $G_q$. 
It is easy to see that
$$
|J_q|=
\begin{cases}
q^3, &  \text{if $q$ is odd},\\
\left(\frac{q}{2}\right)^3, & \text{if $q$ is even}
\end{cases}
$$
and $G_q/J_q \cong \ZZ/2\ZZ \times \ZZ/2\ZZ$ 
(there are 4 cosets: $\{J_q, \overline{x}J_q, \overline{y}J_q, \overline{xy}J_q\}$
and the cosets  $\overline{x}J_q, \overline{y}J_q, \overline{xy}J_q$ have order 
$2$ in $G_q/J_q$). Therefore, $|\PB_3:\K^{(q)}|=|\F_2/\K_{\F_2}^{(q)}|=4|J_q|$. Thus,
$$
|\PB_3:\K^{(q)}|=
\begin{cases}
4q^3, & \text{if $q$ is odd},\\
4\left(\frac{q}{2}\right)^3, & \text{if $q$ is even.}
\end{cases}
$$

Therefore, for an odd integer $n$, we have $|\PB_3 : \K^{(n)}|=4n^3=4\left(\frac{2n}{2}\right)^3=|\PB_3:\K^{(2n)}|$. 
Hence $|\K^{(n)}:\K^{(2n)}|=1$ and we have the desired equality  $\K^{(2n)}=\K^{(n)}$.
\end{proof}

We can now describe the structure of the dihedral poset: 
\begin{prop}
\label{prop:structure-of-Dih}
Let $n, q \ge 3$. Then
$$
\K^{(q)} \subset \K^{(n)} \iff n \mid \lcm(q, 2).
$$
\end{prop}
\begin{proof} Let us take care of the implication ``$\Leftarrow$''.
If $q$ is even, then we have $n \mid q$, and the desired 
set inclusion follows from Remark \ref{rem:Kq-Kn}. 
If $q$ is odd, then $n \mid 2q$ and 
we obtain $\K^{(q)}=\K^{(2q)}\subset \K^{(n)}$ from 
the same remark.

To prove the implication ``$\Rightarrow$'', we 
note that $x_{12}^{K_{\ord}^{(q)}} \in \K^{(q)} \subset \K^{(n)}$. 
Then $(1,1,1)=\psi_n(x_{12}^{K_{\ord}^{(q)}}) = (r^{K_{\ord}^{(q)}},1,1) \in D_n \times D_n \times D_n$, 
and hence $n \mid K_{\ord}^{(q)}$.
\end{proof}

\bigskip

To describe the set of $\GT$-shadows with the target $\K^{(n)}$, we need the following 
statement:
\begin{prop}
\label{prop:Vadym}
For every $n \in \ZZ_{\ge 3}$, the commutator subgroup
$[G_n,G_n]$ of $G_n := \lan \ol{x}, \ol{y} \ran$ consists of elements of the form 
\begin{equation}
\label{elements-C}
(r^{2 n_1}, r^{2 n_2}, r^{2 n_3}), \qquad
(n_1, n_2, n_3) \in (2\ZZ)^3 ~\txt{ or }~
(n_1, n_2, n_3) \in (2\ZZ+1)^3
\end{equation}
i.e. $n_1, n_2, n_3$ are either all even integers or all odd integers. 
\end{prop}
The proof of Proposition \ref{prop:Vadym} is somewhat technical
and it is given in Appendix \ref{app:commutator-subgroup}.

\begin{remark}  
\label{rem:4-mid-n-4-nmid-n}
In Proposition \ref{prop:Vadym}, it makes sense to consider 
integers $n_1, n_2, n_3$ modulo $\ord(r^2)$. Moreover, it makes sense to impose the condition
$$
(n_1, n_2, n_3) \in (2\ZZ)^3 ~\txt{ or }~
(n_1, n_2, n_3) \in (2\ZZ+1)^3
$$
only in the case when $4 \mid n$. 
If $n \in 4 \ZZ + 2$ or if $n$ is odd, then 
\begin{equation}
\label{comm-G-n-odd-or}
[G_n, G_n] = \lan r^2 \ran \times  \lan r^2 \ran \times  \lan r^2 \ran. 
\end{equation}
Indeed, if $n = 4t +2$, then $\ord(r^2)= 2t + 1$. 
Hence $\lan r^4 \ran = \lan r^2 \ran$ and identity 
\eqref{comm-G-n-odd-or} follows from the inclusion 
$$
\lan r^4 \ran \times \lan r^4 \ran \times \lan r^4 \ran \subset 
[G_n, G_n].
$$
If $n$ is odd, then the proof of identity \eqref{comm-G-n-odd-or}
is easier and we leave it to the reader.
\end{remark}

\section{The description of $\GT(\K^{(n)})$}
\label{sec:GT-Kn}

Due to \eqref{conj-by-D} and normality of $\K^{(n)}$ in $\B_3$, 
\begin{equation}
\label{theta-K-n-F2}
\te(\K^{(n)}_{\F_2}) = \K^{(n)}_{\F_2}, \qquad \forall~~ n \ge 3, 
\end{equation}
where $\te$ is the automorphism of $\F_2$ defined in \eqref{theta}. 

Due to identities \eqref{conj-by-si1} and \eqref{conj-by-si2}, we have 
\begin{equation}
\label{conj-by-si1si2}
(\si_1 \si_2) x_{12} (\si_1 \si_2)^{-1} = x_{23}, 
\qquad
(\si_1 \si_2) x_{23} (\si_1 \si_2)^{-1}  = x_{23}^{-1} x_{12}^{-1} c.
\end{equation}

Using \eqref{conj-by-si1si2} and $c \in \K^{(n)}$, one can show that
\begin{equation}
\label{tau-K-n-F2}
\tau(\K^{(n)}_{\F_2}) = \K^{(n)}_{\F_2}, \qquad \forall~~ n \ge 3, 
\end{equation}
where $\tau$ is the automorphism of $\F_2$ defined in \eqref{tau}. 

Properties \eqref{theta-K-n-F2} and \eqref{tau-K-n-F2} imply that 
$\te, \tau \in \Aut(\F_2)$ descend to automorphisms of 
the quotient group $\F_2/ \K^{(n)}_{\F_2} \cong G_n$ and to automorphisms 
of the commutator subgroup $[G_n, G_n] \le G_n$. 

Due to the identification $\F_2/ \K^{(n)}_{\F_2} \cong G_n$, the 
above observations about $\te$, $\tau$ and Proposition \ref{prop:simple-hexa}, 
the set $\GT^{\hs}_{pr}(\K^{(n)})$ of charming $\GT$-pairs
is identified with the set of pairs 
$$
(m, g) ~\in~  \{0,1,\dots, K^{(n)}_{\ord} - 1 \} \times [G_n, G_n]
$$
for which $\gcd(2m+1,  K^{(n)}_{\ord}) = 1$,
\begin{equation}
\label{hexa1-g}
g \te(g) = 1
\end{equation}
and 
\begin{equation}
\label{hexa11-m-g}
\tau^2(\ol{y}^m g) \tau(\ol{y}^m g) \ol{y}^m g = 1.
\end{equation}

To solve \eqref{hexa1-g}, we notice that
$$
\te(\ol{z}) = \big( \te(\ol{x} \ol{y}) \big)^{-1}  = 
\big( \ol{y} \ol{x} \big)^{-1}  =
\big(  (rs, r, rs) (r, s, s) \big)^{-1} = (s, rs, r)^{-1} = (s, rs, r^{-1}). 
$$
Hence
\begin{equation}
\label{theta-bar-z-2}
\te(\ol{z}^{2}) = \ol{z}^{-2}. 
\end{equation}

Let $n_1, n_2, n_3 \in \{0,1,\dots, \ord(r^2)-1\}$. Combining \eqref{theta-bar-z-2} with 
$\te(\ol{x}) = \ol{y}$ and $\te(\ol{y}) = \ol{x}$, we conclude that, 
for every $g := (r^{2 n_1}, r^{2 n_2}, r^{2 n_3}) \in [G_n,G_n]$, we have 
\begin{equation}
\label{theta-nice}
\te(r^{2 n_1}, r^{2 n_2}, r^{2 n_3}) = (r^{2 n_2}, r^{2 n_1}, r^{-2 n_3}). 
\end{equation}
Moreover, since $\tau(x):= y$, $\tau(y):= z$ and $\tau(z) =x$, we have
\begin{equation}
\label{tau-nice}
\tau(r^{2 n_1}, r^{2 n_2}, r^{2 n_3}) = (r^{2 n_3}, r^{2 n_1}, r^{2 n_2}).
\end{equation}

Using \eqref{theta-nice}, we see that $g := (r^{2 n_1}, r^{2 n_2}, r^{2 n_3}) \in [G_n,G_n]$ satisfies 
\eqref{hexa1-g} if and only if
$$
n_1 + n_2 \equiv 0 ~\mod~ \ord(r^2). 
$$

Let us now consider 
$$
m \in \{0,1,\dots, K^{(n)}_{\ord} - 1 \}, \qquad 
\gcd(2m+1, K^{(n)}_{\ord}) = 1 
$$
and assume that $m$ is odd. 

Setting 
$$
g := (r^{2 k}, r^{-2 k}, r^{2 t}),
$$
unfolding the right hand side of  \eqref{hexa11-m-g} and using \eqref{tau-nice},
we get (recall that $m$ is odd):
$$
(r^m,s,s) (r^{-2k}, r^{2t}, r^{2k})
(r^2 s, r^{-1} s, r^m) (r^{2t}, r^{2k}, r^{-2k})
(rs, r^m, rs) (r^{2k}, r^{-2k}, r^{2t}) 
$$
$$
= (r^m,s,s) 
(r^2 s, r^{-1} s, r^m)
(rs, r^m, rs) \, 
(r^{-2k}, r^{-2t}, r^{-2k})
 (r^{-2t}, r^{2k}, r^{2k})
 (r^{2k}, r^{-2k}, r^{2t}) 
$$
$$
= (r^{m+1}, r^{m+1}, r^{-m-1}) 
(r^{-2t}, r^{-2t}, r^{2t}) = (r^{m+1-2t}, r^{m+1 -2t}, r^{2t-m-1}). 
$$

Thus a pair 
$$
\big( m, (r^{2 k}, r^{-2 k}, r^{2 t}) \big)
$$
satisfies \eqref{hexa11-m-g} if and only if 
$$
m+1 \equiv 2t ~\mod~ \ord(r). 
$$

Notice that, if $4 \mid n$, then we have the additional condition\footnote{See Proposition 
\ref{prop:Vadym} and Remark \ref{rem:4-mid-n-4-nmid-n}.} on $k$ and $t$: $k \equiv t \mod 2$. 
This condition is equivalent to the congruence 
$\displaystyle k \equiv \frac{m+1}{2} ~\mod~ 2$. 
Thus we arrive at the following statement:
\begin{prop}  
\label{prop:GT-pr-charm}
Let $n \in \ZZ_{\ge 3}$ and 
$$
\cX_n := \big\{m : m \in \{ 0 ,1,\dots, K_{\ord}^{(n)} -1 \} ~|~ \gcd(2m+1,  K_{\ord}^{(n)}) = 1\big\},
$$
\begin{equation}
\label{varkappa}
\varka(m):=
\begin{cases}
m+1, & \text{if }~ 2 \nmid m, \\
-m,  & \text{if }~ 2 \mid m.
\end{cases}
\end{equation}
Then 
$$
\GT_{pr}^{\hs}(\K^{(n)}) =
\begin{cases}
\big\{ (m, (r^{2 k}, r^{-2 k}, r^{\varkappa(m)})) ~|~ m \in  \mathcal{X}_n,  ~k \in \ZZ, ~ k \equiv \frac{\varkappa(m)}{2} \text{ mod 2} \big\} 
& \text{if } 4 \mid n, \\[0.18cm]
\big\{ (m, (r^{2 k}, r^{-2 k}, r^{\varkappa(m)})) ~|~ m \in  \mathcal{X}_n,  ~k \in \ZZ \big\} & \text{if } 4 \nmid n.
\end{cases}
$$
\end{prop}  
\begin{proof}
For odd $m$, the desired statement is proved right above the proposition. 
The case when $m$ is even is easier and we leave it to the reader. 
\end{proof}

The following lemma plays an important role in describing the connected 
component of $\K^{(n)}$ in the groupoid $\GTSh$:
\begin{lem}
\label{GTKn-tmf-onto}
For every $(m, g) \in \GT_{pr}^{\hs}(\K^{(n)})$, 
\begin{equation}
\label{ker-is-Kn}
\ker(T_{m,g}^{\PB_3}) = \K^{(n)}
\end{equation}
and the homomorphism $T_{m, g}^{\PB_3}$ is surjective. 
\end{lem}
\begin{proof}
Since $\psi_n(c)= (1,1,1)$, we have $\PB_3/\K^{(n)} \cong \F_2/\K^{(n)}_{\F_2}$. 
Hence we may also identify the quotient group $\PB_3/\K^{(n)}$ with the subgroup 
$G_n \le D_n^3$ generated by $(r, s, s)$ and $(rs, r, rs)$.

Consider the faithful action of $D_n$ on the set $\ZZ/n\ZZ$: $r(\overline{j})=\overline{j}+\overline{1}$, 
$s(\overline{j})=-\overline{j}$, which defines an injection $D_n \rightarrow S_n$. 

Let us prove that, for every 
$(m, g) \in \GT_{pr}^{\hs}(\K^{(n)})$, 
there exists a triple $(h_1,h_2,h_3) \in S_n^3$
(depending on $(m,g)$) such that 
\begin{equation}
\label{hhh-conjugates}
\ol{x}^{2m+1}=(h_1,h_2,h_3)\, \ol{x} \, (h_1^{-1},h_2^{-1},h_3^{-1}), 
\quad  g^{-1} \ol{y}^{2m+1} g=
(h_1,h_2,h_3) \, \ol{y}\, (h_1^{-1},h_2^{-1},h_3^{-1}).
\end{equation}

A direct calculation shows that 
$$
\ol{x}^{2m+1} = (r^{2m+1}, s, s), 
\qquad
g^{-1} \ol{y}^{2m+1} g = 
(r^{1-4k} s, r^{2m+1}, r^{1 - 2 \varka(m)} s),
$$
where $g = (r^{2k},r^{-2k},r^{\varka(m)})$.

Consider the bijection $\ZZ/n\ZZ \to \ZZ/n\ZZ$ that sends $\ol{j}$ to $(2\ol{m}+1) \cdot \ol{j}$ and let $b$ be the corresponding element of $S_n$.  

Setting $h_1 := r^{-2k-m} b$, $h_2 := b$ and 
$$
h_3 :=
\begin{cases}
b & \txt{if } m \txt{ is even},\\
b s & \txt{if } m \txt{ is odd},
\end{cases}
$$
we get a triple of permutations $(h_1, h_2, h_3)$ for which \eqref{hhh-conjugates} holds. 

Since $\psi_n(c) = (1,1,1)$, identity \eqref{hhh-conjugates} implies that, for each charming 
$\GT$-pair $(m, g)$ with the target $\K^{(n)}$, there exists an inner automorphism 
$\de$ (depending on $(m, g)$) of $S_n^3$ such that
$$
T_{m,g}^{\PB_3} = \de \circ \psi_n\,.
$$
This implies the desired equality $\ker(T_{m,g}^{\PB_3}) = \K^{(n)}$. 

Combining \eqref{ker-is-Kn} with the isomorphism theorem, we conclude that, 
for every $(m, g) \in \GT_{pr}^{\hs}(\K^{(n)})$, 
the order of the subgroup $T_{m,g}^{\PB_3}(\PB_3) ~\le~ \PB_3 /  \K^{(n)}$ 
coincides with the order of the quotient group $\PB_3 /  \K^{(n)}$. 

Thus the homomorphism $T_{m,g}^{\PB_3}: \PB_3 \to \PB_3 /  \K^{(n)}$ is surjective
and the lemma is proved. 
\end{proof}

Combining Proposition \ref{prop:GT-pr-charm} with Lemma \ref{GTKn-tmf-onto}, 
we get an explicit description of the set $\GT(\K^{(n)})$:
\begin{thm}
\label{thm:GT-Kn-set}
For every $n \ge 3$, the set of $\GT$-shadows with the target $\K^{(n)}$ is
\begin{equation}
\label{GT-Kn}
\GT(\K^{(n)}) = 
\begin{cases}
\big\{ (m, (r^{2 k}, r^{-2 k}, r^{\varkappa(m)})) ~|~ m \in  \cX_n, ~k \in \ZZ, ~ k \equiv \frac{\varkappa(m)}{2} \text{ mod 2} \big\} 
& \text{if } 4\mid n, \\[0.3cm]
\big\{ (m, (r^{2 k}, r^{-2 k}, r^{\varkappa(m)})) ~|~ m \in  \cX_n, ~k \in \ZZ \big\}  & \text{if } 4\nmid n,
\end{cases}
\end{equation}
where 
$$
\cX_n := \big\{m : m \in \{ 0 ,1,\dots, K_{\ord}^{(n)} -1 \} ~|~ \gcd(2m+1,  K_{\ord}^{(n)}) = 1\big\}
$$
and the function $\varka$ is defined in \eqref{varkappa}. 
Furthermore, $\K^{(n)}$ is an isolated object of the groupoid $\GTSh$. 
\end{thm}
\begin{proof}
Due to the second statement of Lemma \ref{GTKn-tmf-onto}, every charming $\GT$-pair 
with the target $\K^{(n)}$ is indeed a $\GT$-shadow, i.e.  $\GT(\K^{(n)}) = \GT_{pr}^{\hs}(\K^{(n)})$. Hence 
Proposition \ref{prop:GT-pr-charm} implies the first statement of the theorem.

The first statement of Lemma \ref{GTKn-tmf-onto}
implies that, if the target of a $\GT$-shadow is $\K^{(n)}$, then 
its source is also $\K^{(n)}$. Thus  $\K^{(n)}$ is indeed the only object of its 
connected component in $\GTSh$. 
\end{proof}

Theorem \ref{thm:GT-Kn-set} implies that $\GT(\K^{(n)}) = \GTSh(\K^{(n)}, \K^{(n)})$
and hence $\GT$-shadows with the target $\K^{(n)}$ form a (finite) group. 

We will now consider two elements $\K^{(q)}$ and $\K^{(n)}$ of the of dihedral poset 
such that $\K^{(q)} \le \K^{(n)}$. We will show that every $\GT$-shadow with the target $\K^{(n)}$
lifts to a $\GT$-shadow with the target $\K^{(q)}$:
\begin{thm}
\label{thm:GTKn-redmap-strong}
Let $n, q \in \ZZ_{\ge 3}$ and $\K^{(q)} \le \K^{(n)}$. Then the reduction homomorphism 
$$
\cR_{\K^{(q)}, \K^{(n)}} :  \GT(\K^{(q)}) \to \GT(\K^{(n)})
$$
is surjective.
\end{thm}
\begin{proof} Let us assume that $4 \mid q$. Due to Proposition \ref{prop:structure-of-Dih}, 
$\K^{(q)} \le \K^{(n)}$ implies that $n \mid q$. 

It suffices to consider the sub-case when $q = n p$, where $p$ is prime. 
For the more general case, the result is achieved by composing reduction maps. 

Let $ (m, (r^{2 k}, r^{-2 k}, r^{\varkappa(m)})) \in \GT(\K^{(n)})$.  
We need to show that there exists $z \in \ZZ$ such that
\begin{itemize}
    \item $k \equiv \frac{\varkappa(m+zn)}{2} \mod 2$,
    \item $\gcd(2(m+zn)+1, q)=1$.
\end{itemize}
Then $(m+zn, (r^{2 k}, r^{-2 k}, r^{\varkappa(m+nz)})) \in \GT(\K^{(q)})$ and it 
gets sent to the element
$$
(m, (r^{2 k}, r^{-2 k}, r^{\varkappa(m)})) \in \GT(\K^{(n)}).
$$

Put $z = 4z_1$, then $m+zn = m+4z_1 n \equiv m \mod 4$, therefore by definition of 
$\varkappa: \frac{\varkappa(m+zn)}{2} \equiv k \mod 2$. 
So we are left with $\gcd(2(m+zn)+1,q) =\gcd(2m+8z_1 n +1, pn)$. 
Notice that $2m+1$ is coprime with $n$, therefore $\gcd(2m+1 + 8z_1 n, n)=1$. 
So we need to find an integer $z_1$ such that $\gcd(2m+8z_1 n +1, p)=1$. For this purpose, we need to consider 3 cases: \\ \\
\textit{Case 1.} $p \mid n$. Then the statement is obvious because component $2m+8z_1n+1$ 
is coprime with $n$ and so with $p$. \\
\textit{Case 2.} $p \nmid n, p=2$. Then $2m+8z_1n+1$ is simply odd. \\
\textit{Case 3.} $p \nmid n, p>2$. In this case, $\gcd(p, 8n) = 1$. Hence applying the Chinese Remainder 
Theorem to $\ZZ/ 8np \ZZ$, we see that there exists an integer $z_1$ for which $p \nmid  (2m+8z_1n+1)$.

\bigskip

We proved that the reduction homomorphism $\cR_{\K^{(q)}, \K^{(n)}}$ is surjective in the case 
when $4 \mid q$. 

If $4 \nmid q$ but $q$ is even, then Proposition \ref{prop:structure-of-Dih} still implies that $n \mid q$. 
In this case the proof of surjectivity of $\cR_{\K^{(q)}, \K^{(n)}}$ is easier and we leave it to the reader. 

If $q$ is odd, then Proposition \ref{prop:structure-of-Dih} implies that $n \mid 2q$. 
Due to Proposition \ref{prop:K-2n-n}, we have  $\GT(\K^{(q)}) = \GT(\K^{(2q)})$ and the 
desired statement follows from the surjectivity of the reduction
homomorphism $\cR_{\K^{(2q)}, \K^{(n)}} : \GT(\K^{(2q)}) \to \GT(\K^{(n)})$ established above. 
\end{proof}

Theorem \ref{thm:GTKn-redmap-strong} gives us the first supporting evidence for Conjecture 
\ref{conj:GT-shadows-are-arithmetical} formulated in Section \ref{sec:arithmetical}.  

\subsection{The group structure on $\GT(\K^{(n)})$}
\label{sec:GT-K-n-group}

Every $n \in \ZZ_{\ge 3}$ can be written uniquely in the from 
$$
n = 2^{\al} n_0,
$$
where $n_0$ is an odd positive integer and $\al \in  \ZZ_{\ge 0}$. 
Moreover, if $\al  \in \{0, 1\}$, then $n_0 \ge 3$. 

Due to Proposition \ref{prop:K-2n-n}, $\K^{(n_0)} =  \K^{(2 n_0)}$ for every odd integer $n_0 \ge 3$.
Hence the groups $\GT(\K^{(n_0)})$ and $\GT(\K^{(2 n_0)})$ coincide and we may only consider 
the case of even $n$ (i.e. $\al \ge 1$). For every even $n$, we have $K^{(n)}_{\ord} = n$; 
we also set $n_1 : = n/2$ and observe that $n_1 : = \ord(r^2)$ in $D_n$.

Recall that the multiplicative group $\big(\ZZ/2n\ZZ \big)^{\times}$ naturally acts 
on $\ZZ/n_1 \ZZ$. Hence we have the semi-direct product: 
\begin{equation}
\label{Zn1-rtimes-units}
\ZZ/n_1 \ZZ\rtimes \big(\ZZ/2n\ZZ \big)^{\times}\,.
\end{equation}

Due to the following proposition, the group $\GT(\K^{(n)})$ can be identified with a 
subgroup of \eqref{Zn1-rtimes-units}:
\begin{prop}
\label{prop:pair-multiplication}
Let $n$ be an even integer $\ge 4$ and $n_1: = n/2$. Let $\nu$ be the following 
function from the set of odd integers to $\ZZ$: 
\begin{equation}
\label{nu}
\nu(u) := 
\begin{cases}
\displaystyle \frac{u - 1}{4}  & \textrm{if} ~~ u \equiv 1 \mod 4 , \\[0.21cm]
\displaystyle \frac{u + 1}{4}  &  \textrm{if} ~~ u \equiv 3 \mod 4.
\end{cases} 
\end{equation}
If $4 \nmid n$, then the formula 
\begin{equation}
\label{vro}
\vro(m, (r^{2 k}, r^{-2 k}, r^{\varka(m)})) := \big(\, k + n_1\ZZ , \, (2m+1) + 2 n \ZZ \, \big)
\end{equation}
defines an isomorphism $\GT(\K^{(n)}) \iso \ZZ/n_1 \ZZ\rtimes \big(\ZZ/2n\ZZ \big)^{\times}$.  
If $4 \mid n$, then formula \eqref{vro} defines an isomorphism from the group 
$\GT(\K^{(n)})$ onto the following index $2$ subgroup of 
$\ZZ/n_1 \ZZ \rtimes \big(\ZZ/2n\ZZ \big)^{\times}$: 
\begin{equation}
\label{H-n}
H_n : = 
\big\{(\ol{k}, \ol{u}) \in \ZZ/n_1 \ZZ\rtimes \big(\ZZ/2n\ZZ \big)^{\times} ~|~ k \equiv \nu(u) \mod 2 \big\}.     
\end{equation}
\end{prop}
\begin{proof}
In this proof and further in the paper, we denote elements of the group
$\ZZ/n_1 \ZZ \rtimes \big(\ZZ/2n\ZZ \big)^{\times}$ by pairs 
$(\ol{k}, \ol{u})$, where $\ol{k}$ is a residue class modulo $n_1 := n/2$ and
$\ol{u}$ is a unit of the ring $\ZZ/2n\ZZ$. Clearly, an integer $u$ that represents 
$\ol{u}$ is odd.  

Since the integer $m$ in considered modulo $K^{(n)}_{\ord} = n$, it is natural to 
consider $2 m + 1$ modulo $2n$. Moreover, since $\gcd(2m+1, n) = 1$, the integer 
$2m + 1$ is also coprime to $2n$ and the map 
$$
m + n \ZZ ~\mapsto~ 2 m + 1 + 2 n \ZZ
$$
is clearly injective. Thus formula \eqref{vro}  defines an injective map 
from $\GT(\K^{(n)})$ to the group in \eqref{Zn1-rtimes-units}. 

To show that $\vro$ is a group homomorphism, we note that $\GT$-shadows
$(m_1, (r^{2 k_1}, r^{-2 k_1}, r^{\varka(m_1)}))$ and 
$(m_2, (r^{2 k_2}, r^{-2 k_2}, r^{\varka(m_2)}))$ are represented by 
pairs $(m_1, f_1), (m_2, f_2) \in \ZZ \times \F_2$, respectively, where
$$
f_1 := x^{2 k_1} y^{-2k_1} z^{\varka(m_1)}, \qquad \txt{and}
\qquad
f_2 := x^{2 k_2} y^{-2k_2} z^{\varka(m_2)}.
$$

Due to  \eqref{composition}, the $\GT$-shadow
$[m_1, f_1] \circ [m_2, f_2]$ is represented by the pair 
$$
\big(2 m_1 m_2 +m_1 + m_2, f_1 E_{m_1,f_1}(f_2) \big).
$$

To find the element of $G_n$ corresponding to the coset 
$f_1 E_{m_1,f_1}(f_2)\K^{(n)}_{\F_2}$, we need to evaluate 
$$
\psi_n\big( f_1 E_{m_1,f_1}(f_2) \big) \in G_n\,.
$$ 

A direct calculation
shows that 
$$
\psi_n( E_{m_1, f_1}(z)) = (r^{2 + 2 m_1 - 4 k_1} s, r^{-(2m_1 + 1)} s, r^{1 - 2\varka(m_1)}).
$$
Since $\varka(m_2)$ is even, we get 
\begin{equation}
\label{psi-n-Em1f1-z-varka-m2}
\psi_n \big(E_{m_1, f_1}(z^{\varka(m_2)}) \big) = 
(1,\, 1,\, r^{(1- 2\varka(m_1)) \varka(m_2)} ) = \psi_n(z^{(1-2\varka(m_1)) \varka(m_2)}).
\end{equation}

Since the images of $x^2$, $y^2$ and $z^2$ in $G_n$ commute with each other, 
identity \eqref{psi-n-Em1f1-z-varka-m2} implies that 
\begin{equation}
\label{psi-n-f}
\psi_n\big( f_1 E_{m_1,f_1}(f_2) \big)  = 
(r^{2 (k_1 + u_1 k_2)},\, r^{-2(k_1 + u_1 k_2)},\, r^{\varka(m_1) + \varka(m_2) - 2 \varka(m_1) \varka(m_2)}),
\end{equation}
where $u_1 := 2 m_1 + 1$. 

To conclude that $\vro$ is a group homomorphism, it remains to show that 
\begin{equation}
\label{unit-identity}
(2m_1+1) (2m_2+1) = 2 m + 1,
\end{equation}
\begin{equation}
\label{varka-identity}
\varka(m_1) + \varka(m_2) - 2 \varka(m_1) \varka(m_2) = \varka(m), 
\end{equation}
where $m := 2m_1 m_2 + m_1 + m_2$.

It is easy to verify identity \eqref{unit-identity} directly.
Here are two ways to prove \eqref{varka-identity}. First, since the pair 
$(m,  f_1 E_{m_1,f_1}(f_2) )$ represents a $\GT$-shadow with the target $\K^{(n)}$, 
equation \eqref{hexa11-m-g} (with $g:= \psi_n(f_1 E_{m_1,f_1}(f_2))$) is satisfied. 
Hence the exponent $\varka(m_1) + \varka(m_2) - 2 \varka(m_1) \varka(m_2)$ in 
\eqref{psi-n-f} must equal $\varka(m)$ modulo $n$ (see the proof of Proposition 
\ref{prop:GT-pr-charm}).

Second, identity \eqref{varka-identity} can be verified directly by considering the four 
cases: $m_1$ and $m_2$ are even, $m_1$ and $m_2$ are odd, $m_1$ is even and $m_2$ is 
odd, $m_1$ is odd and $m_2$ is even. Since \eqref{varka-identity} is symmetric with respect to 
$m_1$ and $m_2$, the fourth case may be ignored. 

Let us verify \eqref{varka-identity} in the case when $m_1$ is even and $m_2$ is odd. 
In this case, $m := 2m_1 m_2 + m_1 + m_2$ is odd, hence $\varka(m) = (2m_1 m_2 + m_1 + m_2) + 1$.
As for the left hand side of  \eqref{varka-identity}, we have 
\begin{align*}
\varka(m_1) + \varka(m_2) - 2 \varka(m_1) \varka(m_2) = - m_1 + (m_2+1) - 2(-m_1) (m_2+1) & =\\
-m_1 + m_2 +1 + 2 m_1 m_2 + 2m_1 =  2 m_1 m_2  + m_1 + m_2 +1.
\end{align*}
Thus, in this case, identity \eqref{varka-identity} holds. 

It remains to prove that 
\begin{itemize}

\item the homomorphism $\vro$ is surjective if $4 \nmid n$, and 

\item the image of $\vro$ is the index $2$ subgroup 
$H_n$ of $\ZZ/n_1 \ZZ\rtimes \big(\ZZ/2n\ZZ \big)^{\times} $ (see 
\eqref{H-n}) if $4 \mid n$.

\end{itemize}

If an integer $u$ represents a unit in $\ZZ/2n\ZZ$, then $u$ is odd. 
Hence there exists $m \in \ZZ$ such that $u = 2 m+1$. Due to the second case 
in \eqref{GT-Kn}, the homomorphism 
$\vro: \GT(\K^{(n)}) \to \ZZ/n_1 \ZZ\rtimes \big(\ZZ/2n\ZZ \big)^{\times}$ is indeed
surjective. 

Let us now take care of the case when $4 \mid n$. 
Again, since an integer $u$ that represents a unit in $\ZZ/2n\ZZ$, 
$u = 2 m + 1$. A direct calculation shows that, an integer $k$ satisfies 
$$
k \equiv \nu(u) \mod 2 
$$ 
if and only if $k \equiv \varka(m)/2 \mod 2$. Thus $\vro\big(\GT(\K^{(n)}) \big)$ indeed 
coincides with $H_n$. 

Finally, a simple counting argument shows that $H_n$ has index $2$ in 
$\ZZ/n_1 \ZZ\rtimes \big(\ZZ/2n\ZZ \big)^{\times}$.
\end{proof}

To get more information about the group $\GT(\K^{(n)})$, we note that, 
due to the Chinese Remainder Theorem, 
\begin{equation}
\label{CRT}
\ZZ/n_1 \ZZ\rtimes \big(\ZZ/2n\ZZ \big)^{\times} \cong
\Aff(\ZZ/n_0\ZZ) \times \big( \ZZ/2^{\al - 1} \ZZ  \rtimes  \big(\ZZ/2^{\al + 1}\ZZ \big)^{\times}  \big).
\end{equation}

Furthermore, if $\al \ge 2$, then Theorem 2' in \cite[Chapter 4, Section 1]{Ireland-Rosen}
implies that 
\begin{equation}
\label{ZZ-mod-2alpha-p-1}
\big(\ZZ/2^{\al + 1}\ZZ \big)^{\times} \cong \lan -\ol{1} \ran  \times \lan \ol{5} \ran, 
\end{equation}
with $\ord(-\ol{1}) = 2$ and $\ord(\ol{5}) = 2^{\al - 1}$. 

The following theorem is a natural corollary of Proposition \ref{prop:pair-multiplication}: 
\begin{thm}
\label{thm:GT-Kn-group}
Let $n \in \ZZ_{\ge 3}$, $n = n_0 \cdot 2^\alpha$, with $n_0$ being odd and $\alpha \in \ZZ_{\ge 0}$. 
Then the following holds:
\begin{equation}
\label{GT-general-eq}
\GT(\K^{(n)}) \cong 
\begin{cases}
\Aff(\ZZ/n_0\ZZ) \times \cZ_2 & \text{ if } \alpha < 2; \\[0.21cm]
\Aff(\ZZ/n_0\ZZ) \times \ti{H}_{\al} & \text{ if } \alpha \ge 2, 
\end{cases}
\end{equation}
where $\cZ_2$ is the cyclic group of order $2$ and $\ti{H}_{\al}$ is the following index $2$ 
subgroup of $\ZZ/2^{\alpha-1}\ZZ \rtimes \big(\ZZ/2^{\alpha+1}\ZZ \big)^{\times}$: 
\begin{equation}
\label{tilde-H-alpha}
\ti{H}_{\al} := \big\{  (\ol{k}, (-\ol{1})^{a} \ol{5}^b) \in 
\ZZ/2^{\alpha-1}\ZZ \rtimes \big(\ZZ/2^{\alpha+1}\ZZ \big)^{\times} ~|~ k \equiv b \mod 2 \big\}.
\end{equation}
\end{thm}
\begin{proof}
Recall that, due to Proposition \ref{prop:K-2n-n}, $\K^{(n_0)} =  \K^{(2 n_0)}$.
Hence, $\GT(\K^{(n_0)}) = \GT(\K^{(2 n_0)})$ for every odd integer $n_0 \ge 3$. 

Due to Proposition \ref{prop:pair-multiplication}, 
$\GT(\K^{(2 n_0)}) \cong \ZZ/n_0\ZZ \rtimes \big( \ZZ/4n_0\ZZ \big)^{\times}$. Hence, 
the Chinese Remainder Theorem implies that
$$
\GT(\K^{(2 n_0)}) \cong \Big(\ZZ/n_0\ZZ \rtimes \big( \ZZ/n_0\ZZ \big)^{\times} \Big) \times 
\big( \ZZ/4\ZZ \big)^{\times}  \cong 
\Aff(\ZZ/n_0\ZZ) \times \cZ_2\,.
$$
Thus the desired statement is proved in the case when $\al < 2$. 

Let us assume that $\al \ge 2$. Due to  Proposition \ref{prop:pair-multiplication}, 
the group $\GT(\K^{(n)})$ is isomorphic to the subgroup 
$H_n \le \ZZ/2^{\al-1} n_0 \ZZ \rtimes \big(\ZZ/2^{\al+1} n_0\ZZ \big)^{\times}$  defined in \eqref{H-n}. 

Let  $\big( k + 2^{\al-1} n_0 \ZZ ,\, u + 2^{\al+1} n_0\ZZ  \big)$ be an element of $H_n$. 
Since the defining condition of $H_n$ is formulated in the ring $\ZZ/2\ZZ$, we only need to focus 
on the residue classes of $k$ (resp. $u$) in $\ZZ/2^{\al-1} \ZZ$ (resp. in $\ZZ/2^{\al+1} \ZZ$). 
So, for the remainder of the proof, we set  
$$
\ol{k} := k + 2^{\al-1} \ZZ, \qquad 
\ol{u} := u + 2^{\al+1} \ZZ. 
$$ 

If $\ol{u} = \ol{5}^b$, then $u \equiv 1 \mod 4$ and
$$
\frac{u-1}{4} \equiv 1+5 + \dots + 5^{b-1} \equiv b \mod 2.
$$
If $\ol{u} = - \ol{5}^b$, then $u \equiv -1 \mod 4$ and
$$
\frac{u + 1}{4} \equiv  - (1+5 + \dots + 5^{b-1}) \equiv b \mod 2.
$$  

Thus, for $\ol{u} =  (-\ol{1})^{a} \ol{5}^b$, the condition $k \equiv \nu(u)$
is equivalent to $k \equiv b \mod 2$.   

Combining this observation with Proposition \ref{prop:pair-multiplication} 
and equation \eqref{CRT}, we conclude that
$$
\GT(\K^{(n)}) \cong \Aff(\ZZ/n_0\ZZ) \times \ti{H}_{\al},
$$
where $\ti{H}_{\al}$ is defined in \eqref{tilde-H-alpha}.

Note that $\ti{H}_{\al} = H_{2^{\al}} \le \ZZ/2^{\alpha-1}\ZZ \rtimes \big(\ZZ/2^{\alpha+1}\ZZ \big)^{\times}$, 
where $H_{n}$ is defined in \eqref{H-n}. Thus Proposition \ref{prop:pair-multiplication} also 
implies that $\ti{H}_{\al}$ is an index $2$ subgroup of
$\ZZ/2^{\alpha-1}\ZZ \rtimes \big(\ZZ/2^{\alpha+1}\ZZ \big)^{\times}$.
\end{proof}

Due to Theorem \ref{thm:GT-Kn-group}, the group $\GT(\K^{(2^{\al})})$ is isomorphic 
to the subgroup $\ti{H}_{\al}$ of $\ZZ/2^{\alpha-1}\ZZ \rtimes \big(\ZZ/2^{\alpha+1}\ZZ \big)^{\times}$
(see \eqref{tilde-H-alpha}).

\section{Non-abelian quotients of $\GTh$ that receive surjective homomorphisms from $G_{\QQ}$}
\label{sec:arithmetical}

In this section, we show that, for each positive integer $n$ divisible by $3$, every
$\GT$-shadow with the target $\K^{(n)}$ satisfies the Lochak-Schneps conditions
(see Corollary \ref{cor:LS-conditions}). Using the surjectivity of the cyclotomic 
character, we find a lower bound for the order 
of the subgroup $\GT_{arith}(\K^{(n)})$ of arithmetical $\GT$-shadows with 
the target $\K^{(n)}$. We prove that every 
$\GT$-shadow with the target $\K^{(2^{\alpha})}$ (with $\al \ge 2$) is arithmetical.
Hence the finite group $\GT(\K^{(2^{\alpha})})$ is naturally a quotient of $G_{\QQ}$ and 
a quotient of $\GTh$. We show that the groups $\GT(\K^{(2^{\alpha})})$ (for $\al \ge 2$)
assemble into an infinite (non-abelian) profinite group and this profinite group receives 
the surjective homomorphism from $G_{\QQ}$ and the surjective homomorphism from the original
version $\GTh$ of the Grothendieck-Teichmueller group. Finally, we identify this 
profinite group with a concrete subgroup of $\Aff(\ZZ_2)$. 

We should mention that there are no direct dependencies between 
Theorem \ref{thm:LS-conditions} (given below) and the rest of this section. However, this result is the 
first instance of applying the ideas of \cite{LochakSchneps-CohomInt} to $\GT$-shadows 
and it is interesting in its own right. Together with 
Theorems \ref{thm:GTKn-redmap-strong} and \ref{thm:main}, Theorem \ref{thm:LS-conditions} 
gives us a supporting evidence for the following conjecture: 
\begin{conj}
\label{conj:GT-shadows-are-arithmetical} 
For every element $\K$ of the dihedral poset $\Dih$, each $\GT$-shadow 
with the target $\K$ is arithmetical. In particular, for every $\K \in \Dih$, the finite group
$\GT(\K)$ is a quotient of $G_{\QQ}$ and a quotient of the 
Grothendieck-Teichmueller group $\GTh$.  
\end{conj}  

First, we address the Lochak-Schneps conditions:
\begin{thm}
\label{thm:LS-conditions}
Let $n$ be a positive integer divisible by $3$. 
For every $[m, f] \in \GT(\K^{(n)})$ there exist $g, h \in \F_2$ such that
\begin{equation}
\label{LS-conditions}
f\, \K^{(n)}_{\F_2} ~ = ~  \te(g)^{-1} g \, \K^{(n)}_{\F_2}\,, \qquad 
 f x^m \, \K^{(n)}_{\F_2}  ~ = ~ 
\begin{cases}
\tau(h)^{-1} h \, \K^{(n)}_{\F_2}, & \text{if}~~ m \equiv 0 \mod 3,\medskip\\
\tau(h)^{-1}  xy h\, \K^{(n)}_{\F_2}, & \text{if}~~ m \equiv -1 \mod 3,
\end{cases}
\end{equation}
where $\te$ and $\tau$ are the automorphisms of $\F_2$ defined in \eqref{theta} and \eqref{tau}.
\end{thm}
\begin{proof} Let $f : = x^{2k}y^{-2k} z^{\varka(m)}$. We set 
$$
g : = 
\begin{cases}
x^{2k}z^{\varka(m)/2}& \txt{if}~~ 4 \mid \varka(m),    \\
y^{2k+2} z^{\varka(m)/2} & \txt{if}~~ 4 \nmid \varka(m). 
\end{cases} 
$$
A direct calculation shows that 
$$
\psi_n \big(\te(g)^{-1} g \big) = (r^{2k}, r^{-2k}, r^{\varka(m)}). 
$$
Thus the first equation in \eqref{LS-conditions} is satisfied.  

Since $n$ is divisible by $3$, so is $K_{\ord}^{(n)}$. 
Since $2 m + 1$ is coprime to $K_{\ord}^{(n)}$, $m$ cannot 
be congruent to $1$ modulo $3$. Hence $m$ may have the following
4 remainders of division by\footnote{Recall that $K_{\ord}^{(n)}$ is always even.} 6: 0, 2, 3, 5. 

If $f : = x^{2k}y^{-2k} z^{\varka(m)}$, then we set 
$$
h : = 
\begin{cases}
x^{2k+m} y^{m}   & \txt{if} ~~  m \equiv 0 \mod 6, \\
x^{2k-1} z^{-m-2}   & \txt{if} ~~  m \equiv 2 \mod 6, \\
x^{-2k-m+1} y^{-m}  & \txt{if} ~~  m \equiv 3 \mod 6, \\
x^{-2k} z^{-m}  & \txt{if} ~~  m \equiv 5 \mod 6.
\end{cases} 
$$

%
%

Let us consider the case when $m \equiv 0 \mod 6$. In this case $m$ is even 
and $m \equiv 0 \mod 3$. Hence we need to check that 
\begin{equation}
\label{residue-0-mod-6}
\psi_n(x^{2k}y^{-2k} z^{\varka(m)} x^m) = \psi_n \big(\tau(h)^{-1} \, h \big),
\end{equation}
where $h = x^{2k+m} y^{m}$. 

Since $m$ is even, $\varka(m) = -m$, we rewrite the left hand side of \eqref{residue-0-mod-6} as follows: 
\begin{align*}
\psi_n(x^{2k}y^{-2k} z^{\varka(m)} x^m) & = (r^{2k}, r^{-2k}, r^{-m})  (r, s ,s)^m =  \\
(r^{2k}, r^{-2k}, r^{-m})  (r^m, 1 , 1) & = (r^{2k + m}, r^{-2k}, r^{-m}).
\end{align*}

As for the right hand side of \eqref{residue-0-mod-6}, we have
$$
\tau(h)^{-1}  = (y^{2k+m} z^{m})^{-1} = z^{-m} y^{-2k-m}\,,
$$
$$
\psi_n \big( \tau(h)^{-1} \big)  = \psi_n(z^{-m}) \psi_n(y^{-2k-m}) = 
(1,1, r^{-m}) (1, r^{-2k-m} ,1)  = (1, r^{-2k-m}, r^{-m})
$$
and 
\begin{align*}
\psi_n \big(\tau(h)^{-1} \, h \big)  =\psi_n \big(\tau(h)^{-1} \big) \, \psi_n( x^{2k+m} y^{m}) & =\\
(1, r^{-2k-m}, r^{-m}) \, (r^{2k+m}, r^m, 1) & =(r^{2k+m}, r^{-2k}, r^{-m}).
\end{align*}
Thus \eqref{residue-0-mod-6} is satisfied.  

Let us consider the case when $m \equiv 5 \mod 6$. In this case $m$ is odd and 
$m \equiv -1 \mod 3$. Hence we need to check that 
\begin{equation}
\label{residue-5-mod-6}
\psi_n(x^{2k}y^{-2k} z^{\varka(m)} x^m) = \psi_n \big(\tau(h)^{-1}\,  xy\, h \big),
\end{equation}
where $h = x^{-2k} z^{-m}$. 

Since $m$ is odd, $\varka(m) = m+1$, we rewrite the left hand side of \eqref{residue-5-mod-6} 
as follows:
\begin{align*}
\psi_n(x^{2k}y^{-2k} z^{\varka(m)} x^m) & = (r^{2k}, r^{-2k}, r^{m + 1}) (r, s, s)^m = \\
(r^{2k}, r^{-2k}, r^{m+1}) (r^m, s, s) & = (r^{2k + m}, r^{-2k} s, r^{m+1} s). 
\end{align*}

As for the right hand side of \eqref{residue-5-mod-6}, we have: 
\begin{align*}
\tau(h)^{-1} & =  (y^{-2k} x^{-m})^{-1} = x^m y^{2k}, \\
\psi_n(xy) & = (r, s, s) (rs, r, rs) = (r^2 s, r^{-1} s, r^{-1}), 
\end{align*}
and 
\begin{align*}
\psi_n \big(\tau(h)^{-1}\,  xy\, h \big)  =  \psi_n( x^m y^{2k})\, \psi_n(xy)\, \psi_n(x^{-2k} z^{-m}) & = \\
(r^m, s, s)(1, r^{2k}, 1) \, (r^2 s, r^{-1} s, r^{-1}) \, (r^{-2k}, 1, 1) \, (r^2 s, r^{-1} s, r^{-m}) & = 
(r^{m+2k}, r^{-2k} s, r^{m+1} s).
\end{align*}
Thus \eqref{residue-5-mod-6} is satisfied. 

We leave the verification of the remaining two cases (when $m \equiv 2 \mod 6$ and $m \equiv 3 \mod 6$) 
to the reader.  
\end{proof}
 
The main result of this paper is the following theorem:
\begin{thm}
\label{thm:main}
Consider an integer $n = 2^\alpha n_0 \ge 3$ with $n_0$ being odd
and let $\GT_{arith}(\K^{(n)})$ be the subgroup of arithmetical $\GT$-shadows 
in $\GT(\K^{(n)})$ (see \eqref{GT-arith-N} in Subsection \ref{sec:Ih-N}). Then
\begin{equation}
\label{lower-bound}
|\GT_{arith}(\K^{(n)})| ~ \ge ~ 
\begin{cases}
2 \phi(n_0)  & \textrm{if} ~~ \al =0 \textrm{ or } \al = 1, \\
2^{2\alpha-2} \phi(n_0) & \textrm{if} ~~ \al \ge 2.
\end{cases}
\end{equation}
In particular for $n_0 = 1$, the group homomorphism 
$$
\Ih_{ \K^{(2^{\alpha})} }: G_{\QQ} \to \GT(\K^{(2^{\alpha})}), \qquad \al \in \ZZ_{\ge 2}
$$
is surjective. 
\end{thm}
\begin{proof} Due to Proposition \ref{prop:K-2n-n}, $\K^{(n_0)} =  \K^{(2 n_0)}$.
Hence, $\GT(\K^{(n_0)}) = \GT(\K^{(2 n_0)})$ for every odd integer $n_0 \ge 3$. 
For this reason, we focus only on the case when $n$ is even.

If $n$ is even, then, due to Proposition \ref{prop:pair-multiplication}, 
$\GT(\K^{(n)})$ is isomorphic a subgroup of 
\begin{equation}
\label{ZZ-by-n1-rtimes-ZZ-by-2n}
\ZZ/n_1\ZZ \rtimes \big(\ZZ/2 n \ZZ \big)^{\times}\,,
\end{equation}
where $n_1 : = 2^{\al - 1} n_0$.

As above, we denote elements of \eqref{ZZ-by-n1-rtimes-ZZ-by-2n} by pairs 
$(\ol{k}, \ol{u})$ with the first entry being a residue class modulo $2^{\al - 1} n_0$
and the second entry being a residue class modulo $2 n = 2^{\al + 1} n_0$. 

Since $(\ol{0}, \ol{1})$ is the identity element of $\GT(\K^{(n)})$ and $(\ol{0}, -\ol{1})$ is the image 
of the complex conjugation, these elements are arithmetical (see Remark \ref{rem:complex-conjugation}).  

Let us denote by $\chi_{n}$ the group homomorphism 
\begin{equation}
\label{chi-n}
\chi_{n}(\ol{k}, \ol{u}) : = \ol{u} \, : \,   \GT_{arith}(\K^{(n)}) \to  
\big(\ZZ/2 n \ZZ \big)^{\times}\,.
\end{equation}

Due to Remark \ref{rem:cR-H-N-and-Ihara}, the restriction 
$\cR_{\K^{(2 n)}, \K^{(n)}}\big|_{ \GT_{arith}(\K^{(2 n)})}$
of the reduction homomorphism $\cR_{\K^{(2 n)}, \K^{(n)}}$
is a homomorphism
from $\GT_{arith}(\K^{(2 n)})$ to $\GT_{arith}(\K^{(n)})$. 
Moreover, it is easy to see that the diagram 
$$
\begin{tikzpicture}
\matrix (m) [matrix of math nodes, row sep=1.5em, column sep=1.5em]
{ \GT_{arith}(\K^{(2 n)})  &  ~ &  \GT_{arith}(\K^{(n)}) \\
~ &  ~\big(\ZZ/2 n \ZZ \big)^{\times} &~ \\};
\path[->, font=\scriptsize]
(m-1-1) edge node[above] {$\cR_{\K^{(2 n)}, \K^{(n)} }$} (m-1-3)
edge node[left] {$\chi_{vir, \K^{(2n)}}~$} (m-2-2)  
(m-1-3) edge node[right] {$~~\chi_{n}$} (m-2-2);
\end{tikzpicture}
$$
commutes. Combining this observation with the surjectivity of  
the homomorphism $\chi_{vir,\, \K^{(2 n)}}$ 
(see Subsection \ref{sec:Ihara-Ih-N}), we conclude that the homomorphism 
$\chi_{n}$ is also surjective.
Hence, for every $\ol{u} \in (\ZZ/2 n \ZZ)^{\times}$, there is 
at least one $\ol{k} \in \ZZ/ n_1 \ZZ$ such that $(\ol{k}, \ol{u}) \in \GT_{arith}(\K^{(n)})$.

For $\ol{u} \in (\ZZ/2n \ZZ)^{\times}$, we set 
$$
T_{\ol{u}} : = \chi_{n}^{-1} (\ol{u}) \subset \GT_{arith}(\K^{(n)}).
$$
Since $T_{\ol{1}}$ is the kernel of the homomorphism 
$\chi_{n} : \GT_{arith}(\K^{(n)}) \to  (\ZZ / 2n \ZZ)^{\times}$ 
and each $T_{\ol{u}}$ is a coset of the subgroup $T_{\ol{1}} \unlhd \GT_{arith}(\K^{(n)})$, 
we conclude that the size $|T_{\ol{u}}|$ of $T_{\ol{u}}$
does not depend on the unit $\ol{u}$. 

Let $\al = 1$, i.e. $n = 2 n_0$ with $n_0$ being odd $\ge 3$. Since  
$$
\big| (\ZZ / 2 n \ZZ)^{\times} \big| = \phi(2 n) = \phi(4 n_0) =  \phi(4) \phi(n_0) = 2 \phi(n_0),
$$
the surjectivity of $\chi_n$ (see \eqref{chi-n}) implies that 
$$
|\GT_{arith}(\K^{(n)})| \ge  2 \phi(n_0).
$$

The case when $\al = 2$ is also straightforward. Indeed, in this case, 
$2 n = 8 n_0$ with $n_0$ being odd $\ge 1$. Since 
$$
\big| (\ZZ / 2n \ZZ)^{\times} \big| = \phi(8 n_0) = \phi(8) \phi(n_0) = 4 \phi(n_0) = 2^{2\cdot 2 - 2} \phi(n_0), 
$$  
the surjectivity of $\chi_n$ (see \eqref{chi-n}) implies that 
the inequality in \eqref{lower-bound} is satisfied for $\al = 2$. 

In the remainder of the proof, we assume that $\al \ge 3$. Hence $\al - 1 \ge 2$. 

Since 
\begin{equation}
\label{order-of-ZZ-mod-2n}
\big| (\ZZ / 2 n \ZZ)^{\times} \big| = \big| (\ZZ / 2^{\al + 1} n_0 \ZZ)^{\times} \big|   = 
\phi(2^{\al + 1} n_0) = 2^{\al} \phi(n_0),
\end{equation}
our goal is to show that the size of each coset $T_{\ol{u}}$ is at least $2^{\alpha-2}$.

First, let us prove that the coset $T_{\ol{5}}$ has at least two elements. 

Let $(\ol{k}, \ol{5}) \in T_{\ol{5}}$. Due to equation \eqref{H-n}, an integer 
$k$ that represents $\ol{k}$ is odd. 

Conjugating $(\ol{k}, \ol{5})$ with $(\ol{0}, -\ol{1}) \in  \GT_{arith}(\K^{(n)})$ we see that 
$(-\ol{k}, \ol{5})$ also belongs to $\GT_{arith}(\K^{(n)})$. If  $\ol{k} = - \ol{k}$, then 
$2 k = 2^{\al - 1} n_0 q$ for some integer $q$.  
Since $\al - 1 \ge 2$ and $k$ is odd, this equality contradicts to the fundamental theorem 
of arithmetic. Thus  $T_{\ol{5}}$ and hence each coset $T_{\ol{u}}$ has a least two elements. 

Since $T_{\ol{1}} = \ker(\chi_n)$ is a subgroup of the cyclic group
$$
\{(\ol{k}, \ol{1})~|~ \ol{k} \in \ZZ/n_1 \ZZ \} \le H_n  \le \ZZ/n_1 \ZZ \rtimes  (\ZZ / 2 n \ZZ)^{\times},
$$
$T_{\ol{1}}$ is also cyclic. 

Let  $(\, \ol{2^{\al_1} k_1}, \ol{1}\,)$ be a generator of $T_{\ol{1}}$. Here $k_1 = 0$
or $k_1$ is an odd (positive) integer. If $k_1 = 0$, then the kernel 
of $\chi_n : \GT_{arith}(\K^{(n)})  \to (\ZZ / 2 n \ZZ)^{\times}$ is trivial. This conclusion 
contradicts to the fact that each coset $T_{\ol{u}}$ has at least two elements.  
Thus $k_1$ is an odd positive integer. Also, note that $\al_1$ must be positive due 
to equation \eqref{H-n}. 

Without loss of generality, we may assume that the integer $2^{\al_1} k_1$ is a divisor of $2^{\al - 1} n_0$. 
Hence $k_1 \mid n_0$ and $1 \le \al_1 \le \al - 1$. 

We conclude that 
\begin{equation}
\label{order-of-T-1}
|T_{\ol{1}}| = 2^{\al - 1 - \al_1}\frac{n_0}{k_1}.
\end{equation}

Again, let $(\ol{k}, \ol{5}) \in T_{\ol{5}}$. Conjugating $(\ol{k}, \ol{5})$ by 
$(\ol{0}, -\ol{1}) \in  \GT_{arith}(\K^{(n)})$, we see that $(-\ol{k}, \ol{5})$ also 
belongs to $T_{\ol{5}}$. 

Since $T_{\ol{5}}$ is a coset of $T_{\ol{1}}$ in $\GT_{arith}(\K^{(n)})$ and 
$T_{\ol{1}}$ is generated by $(\, \ol{2^{\al_1} k_1}, \ol{1}\,)$, we conclude that 
$(-\ol{k}, \ol{5}) = (\, \ol{2^{\al_1} k_1}, \ol{1}\,)^t  \cdot (\ol{k}, \ol{5}) $ for an integer $t$.
Thus 
$$
-k  \equiv k +  2^{\al_1} k_1 t ~\mod~ 2^{\al-1} n_0
$$ 
or equivalently 
\begin{equation}
\label{key-contradiction}
2 k + 2^{\al_1} k_1 t = 2^{\al-1} n_0 q
\end{equation}
for an integer $q$.

Since $k$ is odd and $\al - 1 \ge 2$, $\al_1$ cannot be greater than $1$. 
Otherwise, equality \eqref{key-contradiction} would contradict the fundamental 
theorem of arithmetic.  

We proved that $\al_1 = 1$. Hence equation \eqref{order-of-T-1} implies that 
$$
|T_{\ol{u}}| =  2^{\al - 2}\, \frac{n_0}{k_1} \ge 2^{\al - 2}.
$$

Combining this observation with the calculation in \eqref{order-of-ZZ-mod-2n}, we
conclude that $|\GT_{arith}(\K^{(n)})| \ge 2^{2\alpha-2} \phi(n_0)$. Thus
inequality \eqref{lower-bound} is proved. 

Let us now prove the second statement of the theorem. 

Since $\al \ge 2$, Theorem \ref{thm:GT-Kn-group} implies that 
$\GT(\K^{(2^{\al})})$ is an index $2$ subgroup of 
$\ZZ/2^{\alpha-1}\ZZ \rtimes \big(\ZZ/2^{\alpha+1}\ZZ \big)^{\times}$. 
Therefore, since 
$$
\big| \ZZ/2^{\alpha-1}\ZZ \rtimes \big(\ZZ/2^{\alpha+1}\ZZ \big)^{\times} \big| = 
2^{\al - 1} \cdot \phi(2^{\al+1}) = 2^{\al - 1} \cdot 2^{\al} = 2^{2 \al - 1},
$$
the order of the group $\GT(\K^{(2^{\al})})$ is $2^{2 \al - 2}$. 
Thus inequality \eqref{lower-bound} implies that  
$$
\GT_{arith}(\K^{(2^{\al})}) = \GT( \K^{(2^{\al})} ).
$$

We proved that, for every $\al \in \ZZ_{\ge 2}$, the homomorphism 
$\Ih_{\K^{(2^{\al})}} : G_{\QQ} \to  \GT(\K^{(2^{\al})})$ is surjective. 
\end{proof}

\bigskip

Due to Theorem \ref{thm:main}, we have a new family of non-abelian
finite quotients of the original Grothendieck-Teichmueller group $\GTh$ 
introduced by V. Drinfeld in \cite[Section 4]{Drinfeld}. More precisely, 
\begin{cor}
\label{cor:quotient-of-GTh-G-QQ} 
For every $\al \in \ZZ_{\ge 2}$, the finite group $\GT(\K^{(2^{\alpha})})$ is naturally 
a quotient of $G_{\QQ}$ and a quotient of $\GTh$. If $\al \ge 3$, then the group 
$\GT(\K^{(2^{\alpha})})$ is non-abelian. 
\end{cor}
\begin{proof} The second statement of Theorem \ref{thm:main} implies that the finite 
group $\GT(\K^{(2^{\alpha})})$ is naturally a quotient of $G_{\QQ}$.

Since  $\Ih_{ \K^{(2^{\alpha})} }$ is the composition 
of the homomorphisms 
$$
G_{\QQ} \overset{\Ih}{\tto} \GTh \tto \GTh_{gen} \tto \GT(\K^{(2^{\alpha})})
$$
and the homomorphism $\Ih_{ \K^{(2^{\alpha})} }$ is surjective, we conclude that 
the homomorphism $\GTh \to  \GT(\K^{(2^{\alpha})})$ is also surjective.
Thus the finite group $\GT(\K^{(2^{\alpha})})$ is naturally a quotient of $\GTh$.

Assume that $\al \ge 3$.  Since $\GT(\K^{(2^{\alpha})})$ is isomorphic to the 
subgroup $\ti{H}_{\al}$ (see \eqref{tilde-H-alpha}) of 
$\ZZ/2^{\alpha-1}\ZZ \rtimes \big(\ZZ/2^{\alpha+1}\ZZ \big)^{\times}$, we need to show 
that $\ti{H}_{\al}$ is non-abelian. 

The pairs $(\ol{0}, -\ol{1})$ and $(\ol{1}, \ol{5})$ belong to $\ti{H}_{\al}$ and they 
do not commute in $\ti{H}_{\al}$ since $\ol{1} \neq -\ol{1}$ in $\ZZ/2^{\alpha-1}\ZZ$.  
\end{proof}
\begin{remark}  
\label{rem:non-abelian}
Let $\N \in \NFI_{\PB_3}(\B_3)$. If $\PB_3/\N$ is abelian (or, equivalently, $\F_2/\N_{\F_2}$ is abelian), 
then 
$$
\GT(\N) := \{(m, 1_{\F_2}) ~|~ 0 \le m \le N_{\ord}-1, ~~ \gcd(2m+1, N_{\ord}) =1 \}
$$ 
and every $\GT$-shadow with the target $\N$ is arithmetical due to the surjectivity 
of the cyclotomic character. Thanks to Theorem \ref{thm:main} and its Corollary 
\ref{cor:quotient-of-GTh-G-QQ}, we have a family of isolated objects $\N$ of the groupoid $\GTSh$ 
such that $\PB_3/\N$ is non-abelian, the group $\GT(\N)$ is non-abelian and the version of the Ihara 
homomorphism $\Ih_{\N} : G_{\QQ} \to \GT(\N)$ is surjective.
\end{remark}
\begin{remark}  
\label{rem:open-questions}
In paper \cite{LochakSchneps-Open} (see Question 1.6), P. Lochak and L. Schneps 
asked if anything can be said about finite non-abelian quotients of $\GTh$. The statement 
about $\GTh$ in Corollary \ref{cor:quotient-of-GTh-G-QQ} is a promising step 
in addressing this question. This statement also sheds light on \cite[Question 1.7]{LochakSchneps-Open} 
and it would be very interesting to continue tackling both Question 1.6 and Question 1.7 
from  \cite{LochakSchneps-Open} using $\GT$-shadows. 
We should also mention that Corollary \ref{cor:quotient-of-GTh-G-QQ}
resolves the natural version of \cite[Question 4.7]{GTShadows} for the gentle 
version $\GTh_{gen}$ of the Grothendieck-Teichmueller group  $\GTh$.
\end{remark}

\subsection{A non-abelian infinite profinite quotient of $\GTh$}
\label{sec:infinite-quotient}

Recall \cite[Section 5]{GTgentle} that the assignment 
$$
\N \in \NFI_{\PB_3}^{isolated}(\B_3) \mapsto \GT(\N)
$$ 
upgrades to the functor $\ML$ (called the Main Line functor) from the poset 
$\NFI_{\PB_3}^{isolated}(\B_3)$ of isolated objects of the groupoid $\GTSh$ to the category 
of finite groups (see Remark \ref{rem:ML}).
Moreover, due to \cite[Theorem 5.2]{GTgentle}, the natural map 
from $\GTh_{gen}$ to $\lim(\ML)$ is an isomorphism of topological groups. 

In this section, we will show that the finite groups $\GT(\K^{(2^{\al})})$ for $\al \in \ZZ_{\ge 2}$
assemble into an infinite profinite group and the natural homomorphism 
from $G_{\QQ}$ to this group is surjective. 

For this purpose, we denote by $\Dih_2$ the following subposet of $\NFI_{\PB_3}^{isolated}(\B_3)$: 
\begin{equation}
\label{Dih-2}
\Dih_2 := \{\K^{(2^{\al})}~|~ \al \in \ZZ_{\ge 2} \} \subset \NFI_{\PB_3}^{isolated}(\B_3).
\end{equation}
 
Since each element $\K \in \Dih_2$ is an isolated object of the groupoid $\GTSh$,  
the formula
\begin{equation}
\label{PR-Dih-2}
\PR^{\Dih_2}(\hat{m}, \hat{f})(\K) :=  
\big( \hcP_{K_{\ord}}(\hat{m}), \hcP_{\K_{\F_2}}(\hat{f}) \big), \qquad \K \in \Dih_2
\end{equation}
defines a group homomorphism $\PR^{\Dih_2}$ from $\GTh_{gen}$ to the limit 
$$
\lim\big( \ML \big|_{\Dih_2} \big)
$$
of the functor $\ML\big|_{\Dih_2}$. Note that the poset $\Dih_2$ is totally ordered.
In fact, due to Proposition \ref{prop:structure-of-Dih}, the assignment 
$\al \mapsto \K^{(2^{\al})}$ defines an isomorphism of posets $\ZZ_{\ge 2} \iso \Dih_2$, 
where the set $\ZZ_{\ge 2}$ is considered with the opposite order.  

Since the $\cR_{\K_1, \K_2} : \GT(\K_1) \to \GT(\K_2)$ is surjective for every 
$\K_1, \K_2 \in \Dih_2$ with $\K_1 \le \K_2$ and each $\GT(\K)$ is a finite group 
considered with the discrete topology, the argument given in the proof of 
\cite[Proposition 1.1.4]{RZ-profinite} implies that the natural group homomorphism 
\begin{equation}
\label{from-lim-ML-Dih2-to-GT-K}
\lim\big( \ML \big|_{\Dih_2} \big) \to \GT(\K)
\end{equation} 
is surjective for every $\K \in \Dih_2$. 

Combining this observation with $\displaystyle \lim_{\al \to \infty} \, |\GT(\K^{(2^{\al})}) | = \infty$, we conclude 
that $\lim\big(  \ML\big|_{\Dih_2} \big)$ is an {\it infinite} profinite group. The surjectivity of 
the homomorphism in \eqref{from-lim-ML-Dih2-to-GT-K} and the second statement of Corollary \ref{cor:quotient-of-GTh-G-QQ} 
also imply that the group $\lim\big(  \ML\big|_{\Dih_2} \big)$ is non-abelian. 

Composing the homomorphism $\PR^{\Dih_2}$ with the Ihara embedding $\Ih : G_{\QQ} \into \GTh_{gen}$, 
we get the group homomorphism 
\begin{equation}
\label{Ih-Dih-2}
\Ih^{\Dih_2} : G_{\QQ} \to \lim\big( \ML \big|_{\Dih_2} \big).
\end{equation}

Using Theorem \ref{thm:main}, we can prove that 
\begin{thm}  
\label{thm:infinite-quotient}
The group homomorphism in \eqref{Ih-Dih-2} is surjective.
Hence the infinite profinite group $\lim\big( \ML\big|_{\Dih_2} \big)$ is a quotient of $G_{\QQ}$
and a quotient of the original Grothendieck-Teichmueller group $\GTh$. 
\end{thm}  
\begin{proof} Let $\K_1, \K_2 \in \Dih_2$ and $\K_1 \le \K_2$.
Due to Remark \ref{rem:cR-H-N-and-Ihara}, the diagram 
$$
\begin{tikzpicture}
\matrix (m) [matrix of math nodes, row sep=1.5em, column sep=1.5em]
{~  & G_{\QQ} & ~ \\
\GT( \K_1 ) &  ~ & \GT(\K_2) \\};
\path[->, font=\scriptsize]
(m-1-2) edge node[above] {$\Ih_{\K_1}~~~~~$} (m-2-1)
edge node[above] {$~~~~~\Ih_{\K_2}$} (m-2-3)  
(m-2-1) edge node[above] {$\cR_{\K_1, \K_2}$} (m-2-3);
\end{tikzpicture}
$$
commutes. 

Thus the family of group homomorphisms $\big( \Ih_{\K} \big)_{\K \in \Dih_2}$
is a cone for the functor $\ML\big|_{\Dih_2}$. It is easy to see that the 
homomorphism $\Ih^{\Dih_2}$ from $G_{\QQ}$ to $\lim\big( \ML \big|_{\Dih_2} \big)$ 
comes from the universal property of limits. 

Due to Theorem \ref{thm:main}, each homomorphism in the family 
 $\big( \Ih_{\K} \big)_{\K \in \Dih_2}$ is surjective. Hence \cite[Corollary 1.1.6]{RZ-profinite} implies 
that the group homomorphism $\Ih^{\Dih_2}$ is also surjective. Thus the infinite profinite group 
$$
\lim\big( \ML \big|_{\Dih_2} \big)
$$
is indeed isomorphic to a quotient of $G_{\QQ}$. 

Since the Ihara embedding $\Ih : G_{\QQ} \to \GTh_{gen}$ factors through the 
inclusion map $\GTh \to \GTh_{gen}$ and the homomorphism 
$\Ih^{\Dih_2} : G_{\QQ} \to \lim\big( \ML \big|_{\Dih_2} \big)$ is surjective, we conclude that 
the profinite group $\lim\big( \ML \big|_{\Dih_2} \big)$
is isomorphic to a quotient of the original Grothendieck-Teichmueller group $\GTh$. 
\end{proof}

\bigskip

For completeness, we would like to describe the profinite group 
$\displaystyle \lim\big( \ML \big|_{\Dih_2} \big)$ as a concrete subgroup of 
the group $\Aff(\ZZ_2) := \ZZ_2 \rtimes \ZZ_2^{\times}$. For this purpose, we compose the continuous group 
homomorphism 
\begin{equation}
\label{hcP-8-Aff-ZZ2}
\hcP_{8 \ZZ} : \Aff(\ZZ_2) \to \Aff(\ZZ/8\ZZ) = \ZZ/8\ZZ \rtimes (\ZZ/8\ZZ)^{\times}
\end{equation}
with the natural homomorphism 
\begin{equation}
\label{Aff-ZZ8-to-ZZ2-times-unitsZZ8}
\ZZ/8\ZZ \rtimes (\ZZ/8\ZZ)^{\times}  \to  \ZZ/2\ZZ \rtimes  (\ZZ/8\ZZ)^{\times}\,.
\end{equation}
Note that, since $\ZZ/2\ZZ$ does not have non-trivial automorphisms,  
$\ZZ/2\ZZ \rtimes  (\ZZ/8\ZZ)^{\times}$ is just the abelian group 
$\ZZ/2\ZZ \times  (\ZZ/8\ZZ)^{\times}$. 

Due to Theorem 2' in \cite[Chapter 4, Section 1]{Ireland-Rosen}, 
$(\ZZ/8\ZZ)^{\times} \cong \lan -\ol{1} \ran \times \lan \ol{5} \ran$. 
Moreover, $\lan -\ol{1} \ran \cong \ZZ/2\ZZ$ and 
$\lan \ol{5} \ran \cong \ZZ/2\ZZ$. 

Thus, composing $\hcP_{8 \ZZ}$ with the homomorphism in \eqref{Aff-ZZ8-to-ZZ2-times-unitsZZ8}, 
we get a homomorphism
$$
\Aff(\ZZ_2) \to \ZZ/2\ZZ \times \ZZ/2\ZZ \times \ZZ/2\ZZ\,.
$$ 
Let us denote by 
\begin{equation}
\label{Psi}
\Psi : \Aff(\ZZ_2) \to \ZZ/2\ZZ
\end{equation}
the composition of the above homomorphism with the homomorphism from 
$\ZZ/2\ZZ \times \ZZ/2\ZZ \times \ZZ/2\ZZ$  to $\ZZ/2\ZZ$ that sends $(x_1, x_2, x_3)$ to $x_1 + x_3$. 
It is easy to see that $\Psi$ is a continuous map.

Due to the following proposition, the profinite group 
$\displaystyle \lim\big( \ML \big|_{\Dih_2} \big)$ is a clopen subgroup of 
index $2$ in $\Aff(\ZZ_2)$:
\begin{prop}  
\label{prop:lim-ML-Dih2}
The profinite group 
\begin{equation}
\label{Dih-binary}
\lim\big( \ML \big|_{\Dih_2} \big)
\end{equation}
is isomorphic to the kernel of the continuous homomorphism $\Psi : \Aff(\ZZ_2) \to \ZZ/2\ZZ$ in \eqref{Psi}. 
Moreover, this group is the topological closure of the subgroup of $\Aff(\ZZ_2)$ generated 
by\footnote{We identify $\ZZ$ with its image in the ring $\ZZ_2$ of $2$-adic integers.} 
the elements:
$$
(2,1), ~ (1,5), ~ (0, -1) \in \ZZ_2 \rtimes \ZZ_2^{\times}  = \Aff(\ZZ_2).
$$
\end{prop}   
\begin{proof} Due to Proposition \ref{prop:structure-of-Dih}, 
$\K^{(4)} \supsetneq \K^{(8)} \supsetneq  \K^{(16)} \supsetneq \dots$. 
In other words, the assignment 
$$
\al \mapsto \K^{(2^{\al})}
$$
allows us to identify the poset $\Dih_2$ with $\ZZ^{\opp}_{\ge 2}$, where
$\ZZ^{\opp}_{\ge 2}$ is the set $\ZZ_{\ge 2}$ considered with the opposite order. 

Let us denote by $\cF$ the natural functor from the poset $\ZZ^{\opp}_{\ge 2}$
to the category of finite groups that sends $\al$ to the group 
$$
\cF(\al) : = \ZZ/2^{\alpha-1}\ZZ \rtimes \big(\ZZ/2^{\alpha+1}\ZZ \big)^{\times}\,.
$$

The limit $\lim(\cF)$ of $\cF$ is the subgroup of 
\begin{equation}
\label{product-of-semidirect}
\prod_{\al \in \ZZ_{\ge 2}} \ZZ/2^{\alpha-1}\ZZ \rtimes \big(\ZZ/2^{\alpha+1}\ZZ \big)^{\times}
\end{equation}
which consists of functions 
$$
g :  \ZZ_{\ge 2} \to 
\bigsqcup_{\al \in \ZZ_{\ge 2}} \ZZ/2^{\alpha-1}\ZZ \rtimes \big(\ZZ/2^{\alpha+1}\ZZ \big)^{\times}
$$
satisfying the conditions
$$
g(\al) \in   \ZZ/2^{\alpha-1}\ZZ \rtimes \big(\ZZ/2^{\alpha+1}\ZZ \big)^{\times}, \quad \forall~ \al \ge 2, \qquad
\cP_{\al_1, \al_2} \big(g(\al_1) \big) = g(\al_2), \quad \forall~ \al_1 \ge \al_2\,.
$$

It is easy to see that $\Aff(\ZZ_2)$ is the vertex of a natural cone for the 
functor $\cF$. Moreover, the corresponding (continuous) group homomorphism 
$\Aff(\ZZ_2) \to \lim(\cF)$ is an isomorphism of topological groups. For this reason, 
we identify the profinite group $\Aff(\ZZ_2)$ with $\lim(\cF)$. 

Due to Theorem \ref{thm:GT-Kn-group}, $\ML \big|_{\Dih_2}$ is isomorphic to the functor $\ti{\cF}$ 
that sends $\al \in \ZZ_{\ge 2}$ to the following subgroup of $\cF(\al)$:
\begin{equation}
\label{tilde-cF-alpha}
\ti{\cF}(\al) : = \big\{  (\ol{k}, (-\ol{1})^{a} \ol{5}^b) \in 
\ZZ/2^{\alpha-1}\ZZ \rtimes \big(\ZZ/2^{\alpha+1}\ZZ \big)^{\times} ~|~ k \equiv b \mod 2 \big\} ~\le~ \cF(\al).  
\end{equation}

Thus $\lim(\ti{\cF}) \cong \lim\big( \ML \big|_{\Dih_2} \big)$ is the following subgroup of $\lim(\cF) \cong \Aff(\ZZ_2)$:
$$
\{h \in \lim(\cF) ~|~ h(\al) \in \ti{\cF}(\al)\}. 
$$

Note that an element $h \in \lim(\cF)$ belongs to $\lim(\ti{\cF})$ if and only if 
$h(2) \in  \ti{\cF}(2)$. Thus, due to the above identification of $\lim(\cF)$ with 
$\Aff(\ZZ_2)$,  $\lim(\ti{\cF})$ is precisely the kernel of the homomorphism 
$\Psi : \Aff(\ZZ_2) \to \ZZ/2\ZZ$ (see \eqref{Psi}).

Thus the profinite group $\lim\big(  \ML \big|_{\Dih_2} \big) \cong \lim (\ti{\cF})$ is indeed isomorphic 
to the kernel of the homomorphism $\Psi$. 

Since $\Psi$ is continuous and surjective, we conclude that  $\lim\big(  \ML \big|_{\Dih_2} \big)$ is 
a clopen subgroup of index $2$ in $\Aff(\ZZ_2)$. 

It remains to prove that the profinite group in \eqref{Dih-binary} is the topological closure of the 
subgroup 
$$
\lan  (2,1),  (1,5), (0, -1) \ran  \le \Aff(\ZZ_2). 
$$

Since the images of the elements  $(2,1),  (1,5), (0, -1)$ in 
$\ZZ/2^{\alpha-1}\ZZ \rtimes \big(\ZZ/2^{\alpha+1}\ZZ \big)^{\times}$ belong
to $\ti{\cF}(\al)$ for every $\al \in \ZZ_{\ge 2}$, we conclude that  
$(2,1),  (1,5), (0, -1)$ belong to the subgroup 
$\lim\big(  \ML \big|_{\Dih_2} \big) \le \Aff(\ZZ_2)$.

It is easy to see that the kernels of the homomorphisms (for $\al \ge 2$): 
$$
\Aff(\ZZ_2) \to  
\ZZ/2^{\alpha-1}\ZZ \rtimes \big(\ZZ/2^{\alpha+1}\ZZ \big)^{\times}
$$
form a basis of neighborhoods of the identity element $(0, 1) \in \Aff(\ZZ_2)$. 

Thus it suffices to prove that, for every  $\al \ge 2$, the elements 
$(\ol{2},\ol{1}),  (\ol{1},\ol{5}), (\ol{0}, -\ol{1})$ generate the subgroup 
$\ti{\cF}(\al) \le \ZZ/2^{\alpha-1}\ZZ \rtimes \big(\ZZ/2^{\alpha+1}\ZZ \big)^{\times}$ 
(see \eqref{tilde-cF-alpha}). 

This is straightforward. For example, let us prove that, for every pair of odd (positive) integers $k$ and $q$, 
the element $(\ol{k}, \ol{5}^q)$ belongs to the subgroup of 
$\ZZ/2^{\alpha-1}\ZZ \rtimes \big(\ZZ/2^{\alpha+1}\ZZ \big)^{\times}$ 
generated by the elements $(\ol{2},\ol{1}),  (\ol{1},\ol{5}), (\ol{0}, -\ol{1})$. 

It is easy to show that 
$$
(\ol{1},\ol{5})^q = (\ol{1} + \ol{5} + \dots + \ol{5}^{q-1}, \,  \ol{5}^q), 
\qquad
(\ol{2},\ol{1})^t = (\ol{2t}, \ol{1}) 
$$
and 
\begin{equation}
\label{desired}
(\ol{2},\ol{1})^t \,  (\ol{1},\ol{5})^q  =   \big( \ol{2t} + \ol{1} + \ol{5} + \dots + \ol{5}^{q-1}, \,  \ol{5}^q \big).
\end{equation}

Since $1+ 5 + \dots + 5^{q-1}$ and $k$ are odd, there exists an integers $t$ such that 
$$
k = 2t + 1+ 5 + \dots + 5^{q-1}\,.
$$ 
Thus \eqref{desired} implies that the element $(\ol{k}, \ol{5}^q)$ indeed belongs 
to the subgroup of $\ZZ/2^{\alpha-1}\ZZ \rtimes \big(\ZZ/2^{\alpha+1}\ZZ \big)^{\times}$ 
generated by the elements $(\ol{2},\ol{1}),  (\ol{1},\ol{5}), (\ol{0}, -\ol{1})$.

Using similar arguments, one can show that all the remaining elements 
of $\ti{\cF}(\al)$ belong to the subgroup of 
$\ZZ/2^{\alpha-1}\ZZ \rtimes \big(\ZZ/2^{\alpha+1}\ZZ \big)^{\times}$ generated 
by the elements $(\ol{2},\ol{1}),  (\ol{1},\ol{5}), (\ol{0}, -\ol{1})$.
\end{proof}

\appendix

\section{Proof of Proposition \ref{prop:Vadym}}
\label{app:commutator-subgroup}
In this appendix we present the proof of technical Proposition \ref{prop:Vadym}.

Recall that $n \in \ZZ_{\ge 3}$ and $G_n$ is the subgroup of $D_n \times D_n \times D_n$ generated 
by elements $\ol{x} := (r,s,s)$ and $\ol{y} := (rs, r, rs)$. Our goal is to prove that 
\begin{equation}
\label{comm-subgrp-of-Gn}
[G_n, G_n] =  
\{ (r^{2 n_1}, r^{2 n_2}, r^{2 n_3}) ~|~
(n_1, n_2, n_3) \in (2\ZZ)^3 ~\txt{ or }~
(n_1, n_2, n_3) \in (2\ZZ+1)^3 \}. 
\end{equation}

It is easy to see that the subset 
$$
C_n :=  \{(r^{2 n_1}, r^{2 n_2}, r^{2 n_3}) ~|~ (n_1, n_2, n_3) \in (2\ZZ)^3 ~\txt{ or }~
(n_1, n_2, n_3) \in (2\ZZ+1)^3\} ~\subset~ G_n
$$
is a (normal) subgroup of $G_n$.

Since $G_n$ is generated by two elements and the commutator subgroup $[\F_2, \F_2]$
of $\F_2$ is generated by elements of the form 
\begin{equation}
\label{generators-comm-F2}
[x^t, y^h] = x^t y^h x^{-t} y^{-h}\,, \qquad t,h \in \ZZ, 
\end{equation}
we conclude that $[G_n,G_n]$ is generated by the elements 
\begin{equation}
\label{generators-comm-G}
[\ol{x}^t, \ol{y}^h], \qquad t, h \in \ZZ.  
\end{equation}

We need to consider 4 cases: $t, h$ are even, $t$ is even and $h$ is odd, 
$t$ is odd and $h$ is even, $t, h$ are odd. 

If $t, h$ are even, then $[\ol{x}^t, \ol{y}^h] = (1,1,1)$. 

If $t$ is odd and $h$ is even, then we get 
$$
\ol{x}^t \ol{y}^h \ol{x}^{-t} \ol{y}^{-h} = 
(r^t, s, s) (1, r^h,1) (r^{-t}, s, s) (1,r^{-h}, 1) = (1, [s, r^h], 1) = (1, r^{-2h}, 1).
$$

If $t$ is even and $h$ is odd, then we get 
$$
\ol{x}^t \ol{y}^h \ol{x}^{-t} \ol{y}^{-h} = 
(r^t, 1, 1) (rs, r^h, rs) (r^{-t}, 1, 1) (rs, r^{-h}, rs) = ([r^t, rs], 1,1) = (r^{2t},1,1). 
$$

Finally, if $t$ is odd and $h$ is odd, then we get 
$$
\ol{x}^t \ol{y}^h \ol{x}^{-t} \ol{y}^{-h} = 
(r^t, s, s) (rs, r^h, rs) (r^{-t}, s, s) (rs, r^{-h}, rs) 
$$
$$
= ([r^t, rs], [s, r^h], [s, rs]) = (r^{2t}, r^{-2h}, r^{-2}).
$$

Thus, we conclude that $[G_n,G_n]$ is generated by elements of the form 
\begin{equation}
\label{generators-commG}
\begin{array}{c}
(1, r^{2t}, 1), \qquad  (r^{2t},1,1), \qquad t \in 2\ZZ, \\[0.18cm]
(r^{2n_1}, r^{2n_2}, r^{2}) \qquad n_1, n_2 \in 2\ZZ + 1. 
\end{array}
\end{equation}

Due to this observation, $(1,1,r^4) \in [G_n,G_n]$ and hence 
$$
(1,1, r^{2t}) \in [G_n,G_n], \qquad \forall~ t \in 2 \ZZ.
$$ 
Moreover, $(r^2, r^2, r^2) \in [G_n,G_n]$. 

Since 
$$
(r^{2t},1,1), ~ (1, r^{2t},1), ~ (1,1, r^{2t}) \in [G_n,G_n] 
\qquad \forall ~ t \in 2\ZZ,
$$
and $(r^2, r^2, r^2) \in [G_n,G_n]$, we conclude that $C_n \subset [G_n, G_n]$. 

Since the elements in \eqref{generators-commG} belong to $C_n$, we also have 
the inclusion $[G_n,G_n] \subset C_n$. 
 
Thus the commutator subgroup $[G_n, G_n]$ coincides with $C_n$ and 
Proposition \ref{prop:Vadym} is proved.  \qed

\bigskip

\noindent\textsc{Department of Pure Mathematics and Mathematical Statistics,\\ 
University of Cambridge, Cambridge, \\
England, CB3 0WB, United Kingdom\\
\emph{E-mail address:} {\bf ib538@cam.ac.uk}}

\bigskip

\noindent\textsc{Department of Mathematics,\\
Temple University, \\
Wachman Hall Rm. 638\\
1805 N. Broad St.,\\
Philadelphia PA, 19122 USA \\
\emph{E-mail address:} {\bf vald@temple.edu}}

\bigskip

\noindent\textsc{Department of Mathematics,\\
Massachusetts Institute of Technology\\
77 Massachusetts Ave.\\
Cambridge, MA 02139 USA\\
\emph{E-mail address:} {\bf holikov@mit.edu}}

\bigskip

\noindent\textsc{University of Cambridge, Trinity College,\\
Cambridge, England, CB2 1TQ, United Kingdom\\
\emph{E-mail address:} {\bf vapash27@gmail.com}}

\end{document}